\documentclass[12pt,psamsfonts]{amsart}
\usepackage[english]{babel}
\usepackage[utf8]{inputenc}
\usepackage{amsthm} 
\usepackage{amsmath}
\usepackage{amsfonts}
\usepackage{amssymb}
\usepackage{graphicx}
\usepackage{cite}
\usepackage{geometry}
\usepackage[colorlinks=true, linkcolor=blue, citecolor=blue, urlcolor=blue]{hyperref}
\newtheorem{teorema}{Theorem}

\newtheorem{lema}{Lemma}

\newtheorem{pro}{Proposition}
\newtheorem{cor}{Corollary}
\newtheorem{teoremaLetra}{Theorem}
 
\newcounter{remark}
\begin{document}
\title{On the  Piecewise Holomorphic Systems with Three Zones}

\author{Carlos Vinícius das Neves Silva \and Paulo Ricardo da Silva}

\address{Departamento de Matem\'{a}tica -- Instituto de Bioci\^{e}ncias Letras
	e Ci\^{e}ncias Exatas, UNESP -- Univ Estadual Paulista, Rua C. Colombo, 2265,
	CEP 15054--000 S\~{a}o Jos\'{e} do Rio Preto, S\~{a}o Paulo, Brazil}
\email{carlos.vn.silva@unesp.br and paulo.r.silva@unesp.br}
\date{}


\date{}
\dedicatory{}
\maketitle

\begin{abstract}
We study piecewise-smooth systems with three zones,
$\dot{z} = f_i(z)$, $i = 1,2,3,$ whose discontinuity set $\Sigma$ consists either of a pair of parallel lines or a pair of circles tangent 
to each other internally or externally.
Each $f_i:\overline{\mathbb{C}} \to \overline{\mathbb{C}}$ is assumed to be a holomorphic function.
We establish conditions ensuring the existence of limit cycles in such systems and provide lower bounds 
for the maximum number of limit cycle.
Our approach combines the Melnikov method, local integrability properties of holomorphic systems, 
and the existence of normal forms around 
zeros and poles.
\end{abstract}

\section{Introduction}

Non-smooth dynamical models have attracted significant attention in recent 
decades due to their ability to describe various phenomena observed in engineering, economics, 
biology, and other applied sciences; see, for example, \cite{refb22}. In such systems, abrupt changes in dynamics -- often caused 
by discontinuities -- give rise to complex and rich behaviors that do not appear in smooth systems, 
making them a central topic in current research in dynamical systems.

One of the main problems in the qualitative theory of differential equations, both for smooth and 
non-smooth systems, is the determination of the existence of limit cycles and the establishment 
of bounds for their maximum number, as presented in \cite{refb23}. This question is closely related to Hilbert’s 16th problem, 
originally formulated for planar polynomial differential systems, which asks for the maximum number 
and location of limit cycles in such systems. While the smooth case already poses deep mathematical 
challenges, the piecewise-smooth setting introduces additional difficulties of analytical and geometric nature.

In this work, we address this problem in the context of \textbf{piecewise holomorphic systems} with three zones, i.e.,
\[
\dot{z} = f_i(z), \quad i=1,2,3,
\]
where each $f_i : \overline{\mathbb{C}} \to \overline{\mathbb{C}}$ is a holomorphic function. The holomorphic framework 
offers significant advantages: these systems are locally integrable, their phase portraits near zeros 
and poles are completely classified, and they admit normal forms that facilitate analysis.

Piecewise holomorphic systems have already been studied in different settings. For instance, 
\cite{refb2} provided conditions ensuring the existence of limit cycles for systems separated by a 
straight line; \cite{refb3} investigated bounds on the number of limit cycles in piecewise polynomial 
holomorphic systems; and \cite{refb31} analyzed simultaneous bifurcations of limit cycles. These works, 
together with others in the literature, show the potential of this class of systems for addressing classical problems in a new setting.

Here, we focus on systems where the discontinuity set consists of either \textbf{two parallel lines} or 
\textbf{two circles}. In the case of circles, we emphasize configurations where they are 
\textbf{tangent}---either externally or internally---which introduces additional challenges due 
to the curvature and the local geometry near the tangency point.

For example, in the parallel-line configuration, the discontinuity set is given by
\[
\Im(z) = 1 \quad \text{and} \quad \Im(z) = -1,
\]
dividing the complex plane into three regions:
\begin{equation}
\Sigma_{+}=\{\Im(z) > 1\}, \quad \Sigma_{c}=\{|\Im(z)|<1\}, \quad \Sigma_{-}=\{\Im(z) < -1\},\label{reg+-c}
\end{equation}
with the dynamics on the boundaries governed by Filippov’s convention \cite{refb4}.

Denote \[\mathbb{S}_{1} = \{|z| = 1\}, \quad \mathbb{S}_{2} = \{|z - 2| = 1\}, \quad \mathbb{S}_{3} = \left\{ \left| z - \frac{2}{3} \right| = \frac{1}{3} \right\}.
\]

In the case of  externally tangent circles,  the complex plane is divided into three regions:
\begin{equation} 
\Sigma_{+}^E=\{|z|<1\},\quad \Sigma_{c}^E=\{|z|>1,|z - 2| >1\},\quad \Sigma_{-}^E=\{|z - 2| <1\}
\label{reg+-cE}
\end{equation}
and in the case of internally tangent circles,  the complex plane is divided into three regions:
\begin{equation} 
\Sigma_{+}^I=\left\{\left| z - \frac{2}{3} \right| < \frac{1}{3}\right\},\quad \Sigma_{c}^I=\left\{\left| z - \frac{2}{3} \right| > \frac{1}{3},|z|<1\right\},\quad \Sigma_{-}^I=\{|z | >1\}.
\label{reg+-cI}\end{equation}

A fundamental problem in the theory of non-smooth dynamical systems is to determine the number of limit 
cycles that bifurcate from a center when it is perturbed. While most studies have addressed systems whose 
discontinuity set is a line or an angular sector (i.e., two or more rays meeting at a common point), some works 
have considered more general discontinuity geometries.

In such cases, various techniques have been employed to estimate the number of bifurcating cycles. 
For example, \cite{refb11,refb12} used the averaging method to analyze these bifurcations, whereas 
\cite{refb32,refb33} applied the Melnikov integral to study the number of limit cycles emerging from 
the perturbation of a center.

Here, we study the number of limit cycles that bifurcate from the linear center $\dot{z}=iz$ when perturbed 
by a piecewise polynomial holomorphic system separated by the lines $\Im(z)=1$ and $\Im(z)=-1$. 
We also investigate bifurcations from the center $\dot{z}=-i(z-1)$ when perturbed by two classes of 
piecewise polynomial holomorphic systems: the first having $\mathbb{S}_{1}$ and $\mathbb{S}_{2}$ 
as the discontinuity curves, and the second having $\mathbb{S}_{1}$ and $\mathbb{S}_{3}$.\\
    
\subsection{Our main results.}
In \textbf{Section 3},  we seek conditions that ensure the existence of limit cycles in piecewise holomorphic systems. 
The normal forms associated with holomorphic systems are  
\[1,\quad (a+ib)z, \quad z^{n},\quad
\dfrac{z^{n}}{1+cz^{n-1}}, \quad\mbox{and}\quad \dfrac{1}{z^{n}}. \]

Here we use it to construct examples of limit cycles in 
piecewise holomorphic systems with a straight-line discontinuity. 

We consider discontinuous systems with three zones of the form $\Sigma_+,\Sigma_c,\Sigma_-$ separated by $\Im (z)=\pm 1$, as in \eqref{reg+-c}:
\begin{enumerate}
    \item \quad$\Sigma_+: (a+ib)z$;\quad  $\Sigma_c: iz$;\quad  $\Sigma_-:(c+id)z$. (Proposition \ref{lc1});
    \item \quad$\Sigma_+: (a+ib)z$;\quad $\Sigma_c: z^3$;\quad $\Sigma_-: (c+id)z$. (Proposition \ref{lc3} );
    \item\quad $\Sigma_+: (a+ib)z$;\quad $\Sigma_c: \frac{z^{2}}{1+cz}$;\quad $\Sigma_-: (c+id)z$.  (Proposition \ref{lc31});
    \item \quad$\Sigma_+: (a+ib)z$;\quad $\Sigma_c: \frac{1}{z}$;\quad $\Sigma_-: (c+id)z$.  (Proposition \ref{lc5}).
\end{enumerate}
We begin by analyzing the case in which the discontinuity set consists of two parallel lines. 
Then, by applying the M\"obius transformations we map the parallel lines to two internally tangent circles in the first case, and to two externally 
tangent circles in the second. This allows us to establish conditions guaranteeing the existence 
of limit cycles when the discontinuity set is formed by tangent circles.

In \textbf{Section 4} we proof the following three theorems that analyze the number of limit cycles bifurcating from a 
center when perturbed by the aforementioned classes of piecewise polynomial holomorphic systems.  In what follows  we denote
\begin{eqnarray*}
    h^{+}(z)=\sum_{j=0}^{\ell}(a^{+}_{j}+ib^{+}_{j})z^{j}, \quad
    h^{c}(z)=\sum_{j=0}^{\ell}(a^{c}_{j}+ib^{c}_{j})z^{j}, \quad
    h^{-}(z)=\sum_{j=0}^{\ell}(a^{-}_{j}+ib^{-}_{j})z^{j}.
\end{eqnarray*}

\begin{teoremaLetra}\label{teoa}
Consider the piecewise polynomial holomorphic system of the form
\begin{equation}\label{eq.it1}
\dot{z}=iz+\varepsilon h^{\sigma}(z), \quad z\in\Sigma_{\sigma},\quad  \sigma=+,c,-
\end{equation} where $\ell=4$ and $\Sigma_{\sigma}$ is given by \eqref{reg+-c}.
Then the lower bound for the number of limit cycles bifurcating from the periodic solutions, 
for $\varepsilon$ sufficiently small, is $4$.
\end{teoremaLetra}

\begin{teoremaLetra}\label{teob}
Consider the piecewise polynomial holomorphic system of the form
\begin{equation}\label{eq.it2}
\dot{z}=-i(z-1)+\varepsilon h^{\sigma }(z), \quad z\in\Sigma_{\sigma}^E,\quad  \sigma=+,c,-
\end{equation}
where $\ell=4$ and $\Sigma_{\sigma}^E$ is given by \eqref{reg+-cE}.
Then the lower bound for the number of limit cycles bifurcating from the periodic solutions, 
for $\varepsilon$ sufficiently small, is $5$.
\end{teoremaLetra}

\begin{teoremaLetra}\label{teoc}
Consider the piecewise polynomial holomorphic system of the form
\begin{equation}\label{eq.it3}
\dot{z}=-i(z-1)+\varepsilon h^{\sigma }(z),\quad z\in\Sigma_{\sigma}^I,\quad  \sigma=+,c,-
\end{equation}
where $\ell=3$ and $\Sigma_{\sigma}^I$ is given by \eqref{reg+-cI}.
Then the lower bound for the number of limit cycles bifurcating from the periodic solutions, 
for $\varepsilon$ sufficiently small, is $4$, for the periodic orbits that cross only $\mathbb{S}_{1}$;
or $8$, for the periodic orbits that cross both $\mathbb{S}_{1}$ and $\mathbb{S}_{3}$. 
\end{teoremaLetra}

The proofs of these theorems use suitable M\"obius transformations to rewrite the original 
systems into new systems whose discontinuity set consists of two parallel lines.  
We analyze the Melnikov function associated with the corresponding 
piecewise systems and estimate the number of its zeros, thereby determining the 
number of limit cycles that emerge from the perturbation of the center.

A relevant question in the study of piecewise smoth systems is the determination of the 
maximum number of limit cycles that certain classes of the systems can exhibit. While linear
Hamiltonian systems separated by a line do not admit limit cycles \cite{refb20}, those separated by
two parallel lines can have at most one \cite{refb24}. 

In this context, we study the maximum number of limit cycles that the class of systems defined by

\begin{equation}
(C1): \quad \dot{z}=\dfrac{1}{z_{1}^{\sigma}(z-z_{0}^{\sigma})} , z\in \Sigma_{\sigma},\sigma=+,-; 
\quad\quad  \dot{z}=\dfrac{1}{\sum_{j=0}^{n}z_{j}^{c}z^{j}},z\in \Sigma_{c},
\label{C1}\end{equation}
can have. In this case the upper and lower regions contain saddles points, and the central region 
contains saddle points or unions of hyperbolic sectors. For this case, we obtain the following theorem: 

\begin{teoremaLetra}\label{teod}
    System \eqref{C1} satisfies the following.
    \begin{itemize}
        \item[a)] It has at most $\frac{n(n+1)}{2}$ crossing limit cycles.
        \item[b)] For $n=1$, there exists a system with $1$ crossing limit cycle, 
        which show that the upper bound in $a)$ is attainable in this case. 
        \item[c)] For $n=2$, there exists a system with $2$ crossing limit cycle.
    \end{itemize}
\end{teoremaLetra}
    
When the discontinuity region is a circle, it was shown in \cite{refb7} that the maximum 
number of limit cycles for linear centers is two. In \cite{refb5}, the case of two concentric 
circles was analyzed, where the authors considered limit cycles crossing four points in the 
discontinuity set- two on each circle- and concluded that the maximum number of limit cycles 
for linear centers separated by these circles is three. 

Motivated by these results, we investigate the maximum number of limit cycles than can arise 
in linear holomorphic center when the discontinuity is given by $\mathbb{S}_{1}$ and 
$\mathbb{S}_{2}$ or $\mathbb{S}_{1}$ and $\mathbb{S}_{3}$.

Let $(C2)$ and $(C3)$ be the classes of linear holomorphic centers separated, respectively,  by $\mathbb{S}_{1}$ and 
$\mathbb{S}_{2}$ and by $\mathbb{S}_{1}$ 
and $\mathbb{S}_{3}$. For these classes, we have the following theorem:

\begin{teoremaLetra}\label{teoe}
The following statements hold.
\begin{itemize}
    \item[a)] Every system in $(C2)$ or $(C3)$ has at most $3$ crossing limit cycles.
    \item[b)] There are system in $(C2)$ having $2$ crossing limit cycles.
    \item[c)] There are system in $(C3)$ having $2$ crossing limit cycles.
\end{itemize}
\end{teoremaLetra}

The proof of Theorems E follows the same approach used in Theorems \ref{teob} and \ref{teoc}, 
where Möbius transformations are applied to simplify the discontinuity region and 
make the analysis simpler.

\textbf{Section 2} is devoted to the preliminaries needed for the development of the subsequent sections.
In \textbf{Section 3}, we provide conditions that guarantee the existence of limit cycles. 
In \textbf{Section 4}, we prove Theorems \ref{teoa}, \ref{teob} and \ref{teoc}, and finally, in \textbf{Section 5}, we prove 
Theorems \ref{teod} and \ref{teoe}.

\section{Preliminaries}
Here we present some preliminary results, which will be fundamental for the 
development of the following sections.

\subsection{Melnikov function} We will review some results on Melnikov functions for piecewise smooth systems. 
For more details, see \cite{refb9,refb10,refb30}. Following the steps of the results and adapting them to our case with different configurations of discontinuities, we can apply the ideas with the same structure of the arguments.

Consider a piecewise differential system of the form
\begin{equation}\label{eq1(s2)}
\dot{x}=\dfrac{H^{\sigma}_{y}(x,y)}{R^{\sigma}(x,y)}+\varepsilon f^{\sigma}(x,y),\quad \dot{y}=
-\dfrac{H^{\sigma}_{x}(x,y)}{R^{\sigma}(x,y)}+\varepsilon g^{\sigma}(x,y),
\end{equation}
where $\varepsilon $ is a small parameter, $f^{\sigma}, g^{\sigma}$ and $H^{\sigma}$, are analytic functions, and $\sigma=+$, if  $y \geq 1$; 
$\sigma=c$, if $-1 \leq y \leq 1$ and  $\sigma=-$, if $y \leq -1$.

Note that the straight lines $y=1$ and $y=-1$ divide  the plane into three regions.
Suppose that the unperturbed system (\ref{eq1(s2)}) (when $\epsilon =0$), satisfy the following  assumptions:\\

\textbf{Assumption 1.} There exists an open interval $I=(\alpha ,\beta )$, such that, for $h\in I$, 
exist four points $A^{-}(h)=(a^{-}(h),-1)$, $B^{-}(h)=(b^{-}(h),-1)$, $A^{+}(h)=(a^{+}(h),1)$ and 
$B^{+}(h)=(b^{+}(h),1)$ with $a^{+}(h)\neq a^{-}(h)$ and  $b^{+}(h)\neq b^{-}(h)$, $H^{-}(A^{-}(h))=
H^{-}(B^{-}(h))=h$, $H^{C}(A^{-}(h))=H^{C}(A^{+}(h))$, $H^{+}(A^{+}(h))=H^{+}(B^{+}(h))$ and $H^{C}(B^{+}(h))=H^{C}(B^{-}(h))$.\\

\textbf{Assumption 2.} The unperturbed system has an orbit $L_{h}^{-}$ that  starts at $A^{-}(h)$ 
and ends at $B^{-}(h)$ defined by $H^{-}(x,y)=h$ for $y<-1$. Additionally, the unperturbed system 
has an orbit $L_{h}^{c_{1}}$, that starts at $B^{-}(h)$ and ends  at $B^{+}(h)$, with its trajectory 
defined by $H^{C}(B^{-}(h))=H^{C}(B^{+}(h))$, $-1<y<1$. For $y\geq 1$, the unperturbed system 
has an orbit $L_{h}^{+}$ that starts at  $B^{+}(h)$  and ends at $B^{-}(h)$. Furthermore, in 
the region  $-1<y<1$ the unperturbed system has an orbit $L_{h}^{c_{2}}$  which starts 
at $A^{+}(h)$ and ends at $A^{-}(h)$.

\begin{figure}[!htb]
\centering
\includegraphics[scale=1.1]{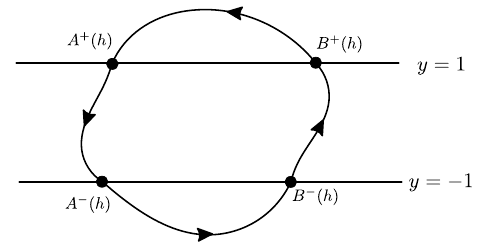}
\caption{The closed orbits of unperturbed system (\ref{eq1(s2)}).}
\label{img1(s2)}
\end{figure}

Under the Assumptions $1$ and $2$, the unperturbed system has a family of periodic 
orbits $L_{h}=L_{h}^{-}\cup L_{h}^{c_{1}}\cup L_{h}^{+}\cup L_{h}^{c_{2}}$ for $h\in J$.
\begin{teorema}\label{melnikov}
Consider system (\ref{eq1(s2)}) with $\varepsilon$ sufficiently small. Under the 
Assumptions $1$ and $2$, the first order Melnikov function can be expressed as
\begin{equation}\label{eq2}
\begin{aligned}
    M(h) &= \dfrac{H_{x}^{-}(A^{-})}{H_{x}^{c}(A^{-})} \int_{L_{h}^{c_{1}}} \left(R^{c}g^{c}
    dx - R^{c}f^{c}dy\right) \\
    &\quad + \dfrac{H_{x}^{-}(A^{-}) H_{x}^{c}(A^{+})}{H_{x}^{c}(A^{-}) H_{x}^{+}(A^{+})} 
    \int_{L_{h}^{+}} (R^{+}g^{+}dx - R^{+}f^{+}dy) \\
    &\quad + \dfrac{H_{x}^{-}(A^{-}) H_{x}^{c}(A^{+}) H_{x}^{+}(B^{+})}{H_{x}^{c}(A^{-}) 
    H_{x}^{+}(A^{+}) H_{x}^{c}(B^{+})} \int_{L_{h}^{c_{2}}} (R^{c}g^{c}dx - R^{c}f^{c}dy )\\
    &\quad + \dfrac{H_{x}^{-}(A^{-}) H_{x}^{c}(A^{+}) H_{x}^{+}(B^{+}) H_{x}^{c}(B^{-})}
    {H_{x}^{c}(A^{-}) H_{x}^{+}(A^{+}) H_{x}^{c}(B^{+}) H_{x}^{-}(B^{-})} \int_{L_{h}^{-}} 
   ( R^{-}g^{-}dx - R^{-}f^{-}dy).
\end{aligned}
\end{equation}

Furthermore, if $M(h)= 0$ and $M'(h)\neq0$ for some $h\in J$, then for  $\varepsilon$ sufficiently 
small, the system (\ref{eq1(s2)}) has a unique limit cycle bifurcating from $L_{h}$.
\end{teorema}
\refstepcounter{remark}
\label{obs.1} 
\noindent\textbf{Remark  \theremark.} 
    The assumptions $1$ and $2$ ensure the existence of a family of periodic orbits that cross 
    the lines $y=1$ and $y=-1$. However, it is also possible to have families of periodic orbits 
    that cross only one of these lines. To consider these case, we assume that:
    \begin{enumerate}
        \item If the periodic orbits cross only $y=1$, 
        \begin{enumerate}
            \item[-]  There exists an open interval $I_{1}=(\alpha_{1},\beta_{1})$ such that, for each 
            $h\in I_{1}$, there exist two points $A_{1}=(a_{1},1)$ and $B_{1}=(b_{1},1)$, 
            with $a_{1}\neq b_{1}$, satisfying $H^{+}(A_{1})=H^{+}(B_{1})=h$ and 
            $H^{c}(A_{1})=H^{c}(B_{1})$.
            \item[-] The unperturbed system has an orbit $M_{h}^{+}$ that starts at $A_{1}$ 
            end ends at $B_{1}$, defined by $H^{+}(x,y)=h$ for $y\geq 1$, and an orbit $M_{h}^{c}$ 
            that starts at $B_{1}$ and ends at $A_{1}$, defined by $H^{c}(x,y)=H^{c}(A_{1})$ 
            for $y<1$. 
        \end{enumerate}
        In this case, the Melnikov function takes the form
      \[
            M_{1}(h)=\frac{H_x^+(A_1)}{H_x^c(A_1)}\int_{M_h^c}(R^cg^cdx-
            R^cf^cdy)+\frac{H_x^+(A_1)H_x^c(B_1)}{H_x^c(A_1)H_x^{+}
            (B_1)}\int_{M_h^+}(R^-g^+dx-R^-f^+dy).
   \]
        \item If the periodic orbits cross only $y=-1$, \begin{enumerate}
            \item[-]  There exists an open interval $I_{2}=(\alpha_{2},\beta_{2})$ such that, for each 
            $h\in I_{1}$, there exist two points $A_{2}=(a_{2},-1)$ and $B_{2}=(b_{2},-1)$, 
            with $a_{2}\neq b_{2}$, satisfying $H^{c}(A_{2})=H^{c}(B_{2})=h$ and 
            $H^{-}(A_{2})=H^{-}(B_{2})$.
            \item[-] The unperturbed system has an orbit $  M_{h}^{c}$ that starts at $A_{2}$ end 
            ends at $B_{2}$, defined by $H^{c}(x,y)=h$ for $y\geq -1$, and an orbit $M_{h}^{-}$ 
            that starts at $B_{2}$ and ends at $A_{1}$, defined by $H^{-}(x,y)=H^{-}(A_{2})$ for $y<-1$. 
        \end{enumerate}
        In this case, the Melnikov function takes the form
      \[
            M_{2}(h)=\frac{H_{x}^{c}(A_{2})}{H_{x}^{-}(A_{2})}\int_{M_{h}^-} (R^-g^-dx-R^-f^-dy)+
            \frac{H_{x}^{c}(A_{2})H_{x}^{-}(B_{2})}{H_{x}^{-}(A_{2})H_{x}^{c}(B_{2})}
            \int_{M_{h}^{c}} (R^{c}g^{c}dx-R^{c}f^{c}dy).
     \]
    \end{enumerate}
Furthermore, if $ M_1(h) = 0 $ and $ M_1'(h) \neq 0 $ for some $ h \in I_{1} $, then for 
$ \varepsilon $ sufficiently small, the system (\ref{eq1(s2)}) has a unique limit cycle 
bifurcating from $ M_{h} $, where $ M_h = M_{h}^{+} \cup M_{h}^{c} $.  

Similarly, if $ M_2(h) = 0 $ and  $M_2'(h) \neq 0$ for some $h \in I_2$, then for $\varepsilon $ 
sufficiently small, the system (\ref{eq1(s2)}) has a unique limit cycle bifurcating from  
$M_{h} $, where $M_{h} = M_{h}^{c} \cup M_{h}^{-}$.

Following, we present an auxiliary result that will be fundamental for estimating the 
number of simple zeros of the Melnikov functions. For more details see \cite{refb6}.
\begin{lema}\label{lem1}
Consider $n+1$ linearly independent functions $f_{i}:U\subset \mathbb{R}\rightarrow \mathbb{R}$, 
$i=0,\cdots n$, then:
\begin{enumerate}
\item[a)] Given $n+1$ arbitrary values $x_{i}\in U$, $i=0,\cdots ,n$ there exist $n+1$ constants 
$C_{i}$, $i=0,\cdots p$ such that
\begin{eqnarray}\label{ar.eq1}
f(x):=\sum_{i=0}^{n}C_{i}f_{i}
\end{eqnarray}
is not the zero function and $f(x_{i})=0$, $i=0,\cdots ,p$;
\item[b) ] Furthermore, if all $f_{i}$ are analytic functions on $U$ and there exists $0\leq j\leq n$
 such that $f_{i} \big|_U$ has a constant sign, it is possible to get an $f$ by (\ref{ar.eq1}), 
 such that it has at least $n$ simple zero in $U$.
\end{enumerate}
\end{lema}

Let $f_{0},\cdots ,f_{n}$ functions, which are differentiable on an interval $I$. Determine if 
$f_{0},\cdots ,f_{n}$ are linearly independent is not a trivial task. A useful tool to determine if 
$f_{0},\cdots ,f_{n}$ are linearly independent is the Wronskian, which is defined by
\begin{eqnarray*}
W(f_0, f_1, \ldots, f_{n})(x)=\hbox{det}(M(f_{0},\cdots ,f_{n-1})(x)),
\end{eqnarray*}
where
\begin{eqnarray*}
M(f_{0},\cdots ,f_{n})(x)= \begin{bmatrix}
f_0(x) & f_1(x) & \cdots & f_{n-1}(x) \\
f_0'(x) & f_1'(x) & \cdots & f_{n-1}'(x) \\
\vdots & \vdots & \ddots & \vdots \\
f_0^{(n)}(x) & f_1^{(n)}(x) & \cdots & f_{n}^{(n)}(x)
\end{bmatrix}.
\end{eqnarray*}

If $W(f_0, f_1, \ldots, f_{n})(x)\neq 0$ for some $x\in I$, then $f_{0},\cdots ,f_{n}$ are 
linearly independent. 

\subsection{Möbius transformations and PWHS}

To study the dynamics of systems separated by tangent circles, we will use Möbius transformations, 
which provide a powerful tool for analyzing the system's structure and behavior. 
A Möbius transformation is a rational function 
$S:\mathbb{C}\setminus\{-\frac{d}{c}\} \to \mathbb{C}\setminus \{\frac{a}{c}\}$ of the form
\begin{eqnarray*}
    S(z)=\dfrac{az+b}{cz+d},
\end{eqnarray*}
where the coefficients $a,b,c$ and $d$ are complex numbers that satisfy $ad-bc\neq 0$. 
The condition $ad-bc\neq 0$ ensures that the function $S$ is not constant.

Möbius transformations satisfy important properties. For example, they are invertible and 
can be expressed as a composition of simpler transformations (translation, 
rotation, homothety and inversion). Moreover, one of their most remarkable 
properties is that the image of a line or circle remains a line or a circle.
Due to these properties, Möbius transformations play a fundamental role in various areas 
of mathematics and, in general, in science. This role is particularly relevant because these 
transformations allow complex problems to be rewritten into simpler equivalent forms, 
facilitating their resolution in different mathematical and applied contexts. An important 
example of how Möbius transformations simplify complex problems can be observed 
in the context of the Dirichlet problem. The main relation between the Dirichlet problem 
and Möbius transformations is that they can be used to map the Dirichlet problem from 
one domain to another (\cite{refb8}). This property enhances the ability to solve 
Dirichlet problems in domains with complex geometry.

To apply Möbius transformations to the study of piecewise holomorphic systems 
separated by two tangent circles, we will perform a change of coordinates that transforms 
the discontinuity regions into simpler domains. This approach allows us to rewrite the 
system in new reference frame, making the analysis of trajectories and dynamic properties more accessible.

\begin{figure}[!htb]
\centering
\includegraphics[scale=0.80]{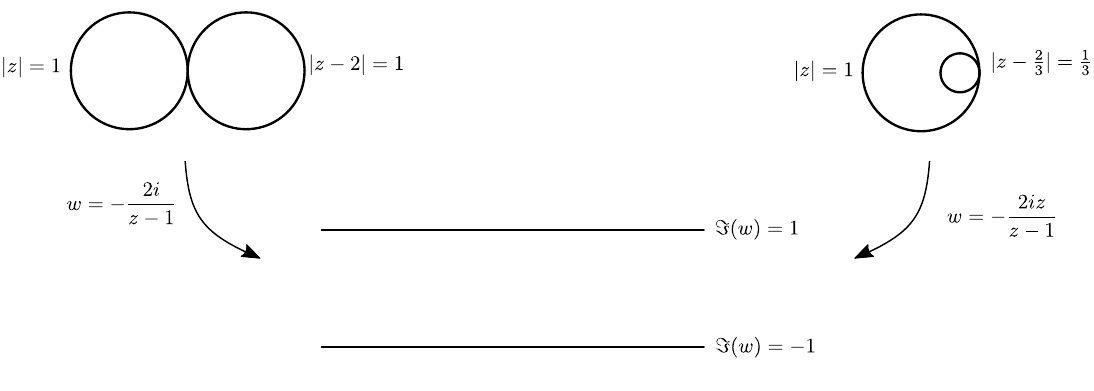}
\caption{Möbius transformations mapping tangent circles to parallel lines, 
simplifying the discontinuity regions.}
\label{img.1}
\end{figure}

In the case of externally tangent circles, the transformation $w=-\frac{2i}{z-1}$ maps 
the circle $|z|=1$ to the line $\Im (w)=1$ and its interior $|z|<1$ to the strip $\Im (w)>1$. 
Additionally, the circle $|z-2|<1$ is transformed into the strip $\Im (w)<-1$. 
The exterior region of the circles $|z|=1$ and $|z-2|=1$ is mapped to the strip $|\Im (w)|<1$.

For the case of internally tangent circles, we use the transformation $w=-\frac{2iz}{z-1}$, 
which maps the circle $|z-\frac{2}{3}|=\frac{1}{3}$ to the line $\Im (w)=1$ and its interior 
$|z-\frac{2}{3}|<\frac{1}{3}$ to the half-plane $\Im (w)>1$. The circle $|z|=1$ is mapped to 
the line $\Im (w)=-1$. Furthermore, the region $|z|<1$ and $|z-\frac{2}{3}|>\frac{1}{3}$ 
is transformed into the strip $|\Im (w)|<1$, while the region $|z|>1$ corresponds to $\Im (w)<-1$.

Thus in both, cases, we transform a system initially separated  by a more complex 
region into an equivalent system where the discontinuity regions assume simpler forms.\\

Consider the system 
\begin{equation}\label{mobiuseq1}\dot{z}=f^{\sigma}(z) \quad \hbox{if} \quad z\in\Sigma_{\sigma},\quad \sigma=+,c,-.
\end{equation}

\begin{pro}
    Let $w=\phi(z)=\frac{az+b}{cz+d}$ be a Möbius transformation. Then the system (\ref{mobiuseq1}) is transformed into
\begin{equation}
\dot{w}=-\dfrac{(cw-a)^{2}}{(bc-ad)} f^{\sigma}\left(\dfrac{b-wd}{cw-a} \right) \quad \hbox{if} \quad w\in\phi (\Sigma_{\sigma}),\quad \sigma=+,c,-.
\end{equation}    
    \end{pro}
    
\begin{proof}
  Note that
  \begin{eqnarray*}
      \phi^{-1}(w)=\dfrac{b-cw}{cw-a},\quad \phi'(z)=\dfrac{ad-cb}{(cz+d)} \quad \hbox{and}\quad \phi'(\phi^{-1}(w))=-\dfrac{(cw-a)^{2}}{cb-ad},
  \end{eqnarray*}
  then
  \begin{eqnarray*}
      \dot{w}=\phi'(\phi^{-1}(w))f^{\sigma }(\phi^{-1}(w))=-\dfrac{(cw-a)^{2}}{cb-ad}f^{\sigma }\left(\dfrac{b-cw}{cw-a} \right).
  \end{eqnarray*}
\end{proof}

To study piecewise holomorphic systems separated by two tangent circles, we will use two 
Mobius transformations and their respective inverses.

\section{Initial examples: Existence of limit cycles for normal forms of PWHS}

In this section, we will find conditions that ensure the existence of limit for piecewise holomorphic systems.

\begin{pro}\label{lc1}
Let $a,b,c,d,x_{0}$ and $x_{1}$ be real numbers such that  $a,c<0$ e $b,d>0$, $x_{0},x_{1}>0$. 
 Then the following  system
\begin{eqnarray}\label{sistema1}
\begin{cases}
\dot{z}=(a+ib)(z-i+x_{1}), \quad \hbox{if} \quad z\in  \Sigma_{+},\\
\dot{z}=iz, \quad \quad \hbox{if} \quad z\in \Sigma_{c},\\
\dot{z}=(c+id)(z+i-x_{0}), \quad \hbox{if}\quad z\in  \Sigma_{-},
\end{cases}
\end{eqnarray}
 has a limit cycle.
\end{pro}

\begin{proof}
    Let $w_{0}=s-i\in  \Sigma_{2}=\lbrace z\in \mathbb{C}; \Im(z)=-1\rbrace $, then we have
\begin{eqnarray*}
    z^{-}(t)&=&(w_{0}+i-x_{0})e^{(c+id)t}-(i-x_{0})=(s-x_{0})e^{ct}(\cos dt+i\sin dt )-(i-x_{0}),
\end{eqnarray*}
is a solution of $\dot{z}=(c+id)(z+i-x_{0})$ with the initial condition  $z^{-}(0)=w_{0}$. Furthermore, we obtain
\begin{eqnarray*}
    z^{-}\left(\frac{\pi }{d}\right)&=&(s-x_{0})e^{\frac{c\pi }{d}}(\cos \pi +i\sin \pi )-(i-x_{0})
								=-(s-x_{0})e^{\frac{c\pi }{d}}+x_{0}-i.
\end{eqnarray*}

Let $z(t)$ be a solution of $\dot{z}=iz$ with the initial condition $z(0)=-(s-x_{0})e^{\frac{c\pi }{d}}+x_{0}-i$. 
By the symmetry of the solutions, which are circles centered at the origin, there exists $t_{0}>0$ 
such that $z(t_{0})=-(s-x_{0})e^{\frac{c\pi }{d}}+x_{0}+i$.

Note that 
\begin{eqnarray*}
z^{+}(t)&=&(-(s-x_{0})e^{\frac{c\pi }{d}}+x_{0}+i-i+x_{1})e^{(a+ib)t}-(-i+x_{1})\\
		&=&(-(s-x_{0})e^{\frac{c\pi }{d}}+x_{0}+x_{1})e^{at}(\cos bt +i\sin bt )-x_{1}+i,
\end{eqnarray*}
is a solution of $\dot{z}=(a+ib)(z-i+x_{1})$, satisfying $z^{+}(0)=-(s-x_{0})e^{\frac{c\pi }{d}}+x_{0}+i$, 
ant thus we obtain

\begin{eqnarray*}
z^{+}\left(\frac{\pi }{b}\right) =((s-x_{0})e^{\frac{c\pi }{d}}-x_{0}-x_{1})e^{\frac{a\pi }{b}}-x_{1}+i.
\end{eqnarray*}

Now, consider $\tilde{z}(t)$, a solution of the equation $\dot{z}=iz$ with the initial condition 
$\tilde{z}(0)=((s-x_{0})e^{\frac{c\pi }{d}}-x_{0}-x_{1})e^{\frac{a\pi }{b}}-x_{1}+i$. 
Again, by the symmetry of the solutions of $\dot{z}=iz$, there exists $\tilde{t_{0}}$ such that $z(\tilde{t_{0}})=((s-x_{0})e^{\frac{c\pi }{d}}-x_{0}-x_{1})e^{\frac{a\pi }{b}}-x_{1}-i$.

Thus, the Poincaré map in a neighborhood of $w_{0}=s-i$ is given by
\begin{eqnarray*}
\Pi (z)=((s-x_{0})e^{\frac{c\pi }{d}}-x_{0}-x_{1})e^{\frac{a\pi }{b}}-x_{1}-i.
\end{eqnarray*}

Finally, we seek the solutions of the equation $\Pi(w_{0})=w_{0}$. 
The number of solutions of this equation corresponds to number of limit cycles of the systems. 
Since the equation $\Pi(w_{0})=w_{0}$ has a unique solution given by 
$\dfrac{-x_{0}e^{\frac{c\pi }{d}+\frac{a\pi }{b}} - (x_{0} + x_{1})e^{\frac{a\pi }{b}} - x_{1}}{1 - e^{\frac{c\pi }{d}+
\frac{a\pi }{b}}}-i$, then the system has  a unique limit cycle.
\end{proof}

Applying the change $w=\frac{z-2i}{z}$ and $w=\frac{z}{z+2i}$ to system (\ref{sistema1}), 
we obtain the following result:

\begin{cor}
    Under the hypotheses of proposition (\ref{lc1}), we have that the systems
\begin{eqnarray}\label{sistema1.1}
\begin{cases}
\dot{w}=\left(\frac{b-ai}{2}\right)(w-1)(-2i+(w-1)(-i+x_{1})), \quad \hbox{if}\quad w\in\Sigma_{+}^{E},\\
\dot{w}=-i(w-1), \quad \hbox{if}\quad w\in \Sigma_{c}^{E},\\
\dot{w}=\left(\frac{d-ci}{2}\right)(w-1)(-2i+(w-1)(i-x_{0}) ), \quad \hbox{if}\quad w\in \Sigma_{-}^{E},
\end{cases}
\end{eqnarray}
and
\begin{eqnarray}\label{sistema1.2}
\begin{cases}
\dot{w}=\left(\frac{b-ai}{2}\right)(w-1)(-2iw+(w-1)(-i+x_{1})), \quad \hbox{if} \quad w\in \Sigma_{+}^{I},\\
\dot{w}=-iw(w-1), \quad \hbox{if}\quad w\in \Sigma_{c}^{I},\\
\dot{w}=\left(\frac{d-ci}{2}\right)(w-1)(-2iw+(w-1)(i-x_{0}) ), \quad \hbox{if}\quad w\in \Sigma_{-}^{I},
\end{cases}
\end{eqnarray}
  has a limit cycle.  
\end{cor}

It is important to note that the conditions provided in the previous proposition are not trivial. 
In fact, by taking $a=c=-1$, $b=d=1$ and $x_{0}=x_{1}=1$, we obtain the existence of a 
limit cycle (as illustrated in figure \ref{img3.1}).
\begin{figure}[!htb]
\centering
\includegraphics[scale=0.56]{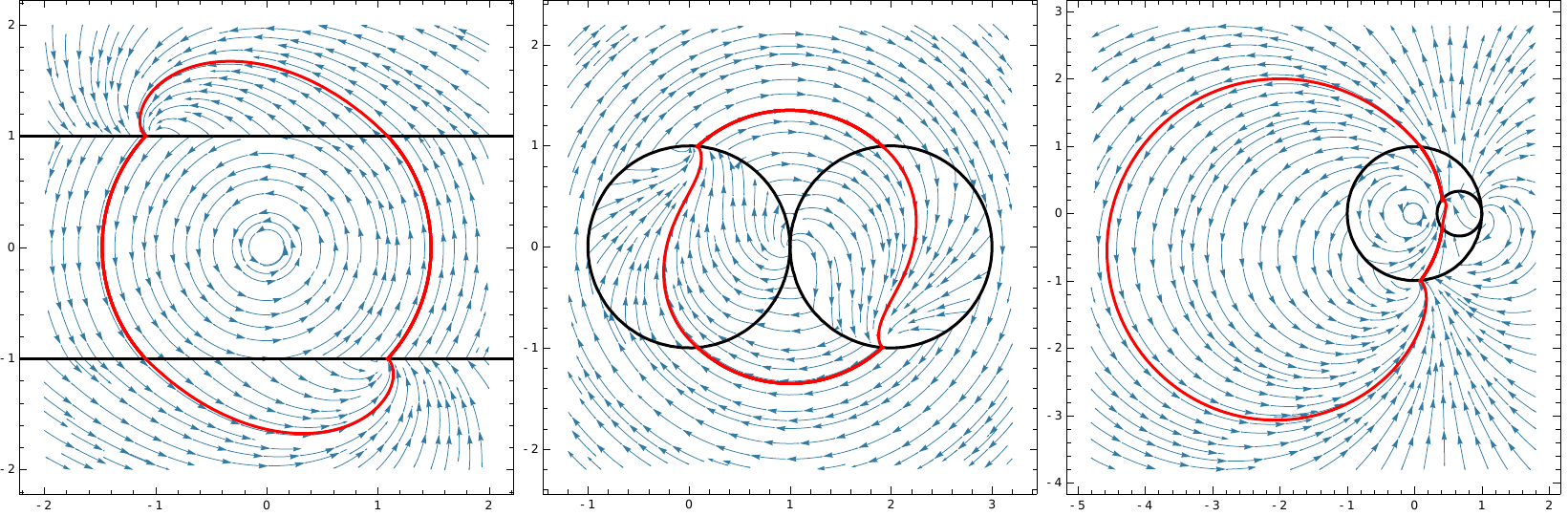}
\caption{Phase portrait of systems (\ref{sistema1}), (\ref{sistema1.1}) and (\ref{sistema1.2}) with $a=c=-1$, $b=d=1$ and $x_{0}=x_{1}=1$ .}
\label{img3.1}
\end{figure}

\begin{pro}\label{lc3}
    Let $a,b,c,d,x_{0}$ and $x_{1}$ be real number such that  $a,c<0$ e $b,d>0$, $x_{0},x_{1}>\sqrt{3}$. Then the following  system 
\begin{eqnarray}\label{sistema3}
\begin{cases}
\dot{z}=(a+ib)(z-i+x_{1}), \quad \hbox{if} \quad z\in \Sigma_{+},\\
\dot{z}=iz^{3}, \quad \quad \hbox{if} \quad z\in \Sigma_{c},\\
\dot{z}=(c+id)(z+i-x_{0}), \quad \hbox{if} \quad z\in \Sigma_{-},
\end{cases}
\end{eqnarray}
 has a limit cycle.
\end{pro}

\begin{proof}
    First, note that the solutions of the equation $\dot{z}=iz^{3}$ are symmetric with respect to the $x$-axis.

    In fact, by writing $\dot{z}=iz^{3}$ in polar coordinates, we obtain that 
    \begin{eqnarray*}  
    \begin{cases}
    \dot{r}=-r^{3}\sin 2\theta,\\ 
    \dot{\theta}=r^{2}\cos 2\theta,
    \end{cases}
    \end{eqnarray*}
    where $z=x+iy=r(\cos \theta +i\sin \theta)$. It follows that
    \begin{eqnarray*}
        \dfrac{dr}{d\theta }=-r\tan 2\theta .
    \end{eqnarray*}

    Therefore, the orbits of this system satisfy:
    \begin{eqnarray*}
        r=k|\cos 2\theta|^{\frac{1}{2}}.
    \end{eqnarray*}
    
    Since the function satisfies $r(-\theta )=r(\theta )$, it follows that the orbits are symmetric with respect to the $x$-axis.

    Now, consider $w_{0}=s-i$, with $s<-\sqrt{3}$. We have that
    \begin{eqnarray*}
        z^{-}(t)=(w_{0}+i-x_{0})e^{(c+id)t}-(i-x_{0})
    \end{eqnarray*}
    is a solution of the equation $\dot{z}=(c+id)(z+i-x_{0})$, with the initial condition $z^{-}(0)=w_{0}$. Also 
    \begin{eqnarray*}
        z^{-}\left( \frac{\pi }{d}\right)=-(s-x_{0})e^{\frac{c\pi }{d}}+x_{0}-i.
    \end{eqnarray*}

    Let $z(t)$ be the solution of $\dot{z}=iz^{3}$ with initial condition $z(0)=-(s-x_{0})e^{\frac{c\pi }{d}}+x_{0}-i$. 
    Due to the symmetry of the solutions with respect to the $x$-axis, there exists $t_{0}$ such that $z(t_{0})=-(s-x_{0})e^{\frac{c\pi }{d}}+x_{0}+i$.

    Observe that 
    \begin{eqnarray*}
    z^{+}(t)=(-(s-x_{0})e^{\frac{c\pi }{d}}+x_{0}+i-i+x_{1})e^{(a+ib)t}-(-i+x_{1})
    \end{eqnarray*}
    is solution of the equation $\dot{z}=(a+ib)(z+i-x_{1})$, with initial condition $z^{+}(0)=z^{+}(0)=-(s-x_{0})e^{\frac{c\pi }{d}}-x_{0}+i$ and 
    \begin{eqnarray*}
    z^{+}\left( \frac{\pi }{a}\right)=((s-x_{0})e^{\frac{c\pi }{d}}-x_{0}-x_{1})e^{\frac{a\pi }{b}}-x_{1}+i.
    \end{eqnarray*}

    Now, consider $\tilde{z}(t)$, a solution of $\dot{z}=iz^{3}$, with initial condition 
    $((s-x_{0})e^{\frac{c\pi }{d}}-x_{0}-x_{1})e^{\frac{a\pi }{b}}-x_{1}+i$. Using the symmetry 
    of the solutions of $\dot{z}=iz^{3}$ with respect to the $x$-axis, we conclude that there 
    exists $t_{1}$ such that $\tilde{t_{1}}=((s-x_{0})e^{\frac{c\pi }{d}}-x_{0}-x_{1})e^{\frac{a\pi }{b}}-x_{1}-i$.

    Therefore, the Poincaré map in a neighborhood of $w_{0}=s-i$ is given by
    \begin{eqnarray*}
    \Pi(s-i)=((s-x_{0})e^{\frac{c\pi }{d}}-x_{0}-x_{1})e^{\frac{a\pi }{b}}-x_{1}-i.
    \end{eqnarray*}

    Since $a,b<0$, the equation $\Pi(w_{0})=w_{0}$ has a unique solution, given by 
    \begin{eqnarray*}
    w_{0}=\dfrac{-x_{0}e^{\frac{c\pi }{d}+\frac{a\pi }{b}}-(x_{0}+x_{1})e^{\frac{a\pi }{b}})-x_{1}}{(1-e^{\frac{c\pi }{d}+\frac{a\pi }{b}})}-i.    
    \end{eqnarray*}
  
    Consequently, the system has a unique limit cycle.
\end{proof}

Although the previous proposition was stated for $\dot{z}=iz^{3}$, it can be extended to 
the case $\dot{z}=iz^{n}$, with $n$ odd. To do so, the values of $x_{0}$ and $x_{1}$ must 
be adjusted according to the value of $n$. In fact, just as in the cubic case, the equation $\dot{z}=iz^{n}$ also exhibits symmetry with respect to  the real axis. The arguments used in the proof remain valid in this more general setting.

\begin{cor}
    Under the hypotheses of proposition (\ref{lc3}), we have that the systems
\begin{eqnarray}\label{sistema3.1}
\begin{cases}
\dot{w}=\left(\frac{b-ai}{2}\right)(w-1)(-2i+(w-1)(-i+x_{1})), \quad \hbox{if}\quad w\in \Sigma_{+}^{E},\\
\dot{w}=\frac{4i}{w-1}, \quad \quad \hbox{if} \quad w\in \Sigma_{c}^{E},\\
\dot{w}=\left(\frac{d-ci}{2}\right)(w-1)(-2i+(w-1)(i-x_{0}) ), \quad \hbox{if} \quad w\in \Sigma_{-}^{E},
\end{cases}
\end{eqnarray}
and
\begin{eqnarray}\label{sistema3.2}
\begin{cases}
\dot{w}=\left(\frac{b-ai}{2}\right)(w-1)(-2iw+(w-1)(-i+x_{1})), \quad \hbox{if}\quad w\in \Sigma_{+}^{I},\\
\dot{w}=\frac{4iw^{3}}{w-1}, \quad \quad \hbox{if} \quad w\in \Sigma_{c}^{I},\\
\dot{w}=\left(\frac{d-ci}{2}\right)(w-1)(-2iw+(w-1)(i-x_{0}) ), \quad \hbox{if}\quad w\in \Sigma_{-}^{I},
\end{cases}
\end{eqnarray}
  has a limit cycle.  
\end{cor}

It is possible to find piecewise holomorphic systems that satisfy the conditions of the 
previous proposition. As an example, take $a=c=-1$, $b=d=1$ and $x_{0}=x_{1}=2$. 
Then, system (\ref{sistema3}) has a limit cycle.

\begin{figure}[!htb] 
\centering
\includegraphics[scale=0.56]{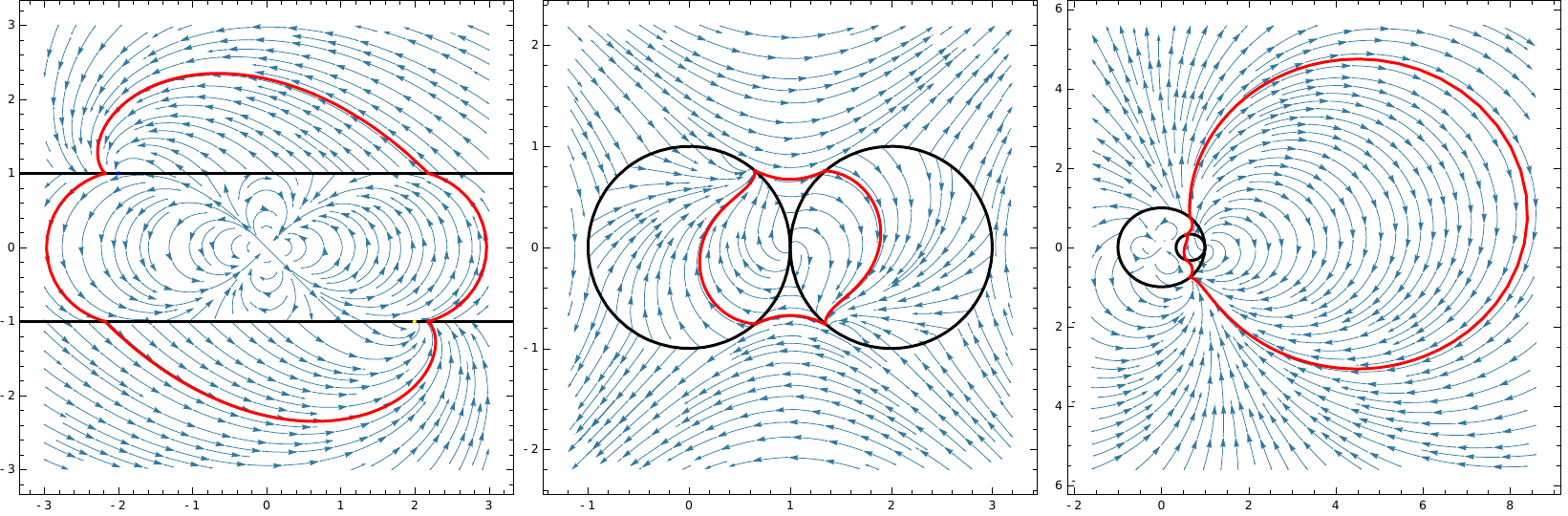}
\caption{Phase portrait of system (\ref{sistema3}), (\ref{sistema3.1}) and (\ref{sistema3.2}) with $a=c=-1$, $b=d=1$ and $x_{0}=x_{1}=2$ .}
\label{img3.3}
\end{figure}

Before presenting the next proposition, we recall a useful result proved in \cite{refb1,refb21}, 
which will play an important role in the proof of the following proposition an later results as well:
\begin{pro}\label{lc2}
  The orbits of $\dot{z}=f(z)$ are contained in the level curves of $\Im (G(z))$, with $G'(z)=\frac{1}{f(z)}$.  
\end{pro}

Now, we study the holomorphic vector field $\frac{z^{2}}{1+z}$, which will appear in the 
central zone of the piecewise-defined system we are about to study. To understand the 
behavior of its trajectories we will make use of the previous proposition.

\begin{pro}\label{lc31}
    Let $a,b,c,d,x_{0}$ and $x_{1}$ be real number such that  $a,c<0$, $b,d>0$, $x_{0}> 1.36$,
    $x_{1}>0.55$. Then the following  system 
\begin{eqnarray}\label{sistema2}
\begin{cases}
\dot{z}=(a+ib)(z-i+x_{1}), \quad \hbox{if} \quad z\in \Sigma_{+},\\
\dot{z}=\dfrac{iz^{2}}{1+z}, \quad \quad \hbox{if} \quad z\in \Sigma_{c},\\
\dot{z}=(c+id)(z+i-x_{0}), \quad \hbox{if} \quad z\in \Sigma_{-},
\end{cases}
\end{eqnarray}
 has a limit cycle.
\end{pro}

\begin{proof}
    Initially, note that the solutions of the equation $\dot{z}=\frac{iz^{2}}{1+z}$ are contained in the level curves of the function 
    \begin{eqnarray*}
        H(x,y)=-\frac{x}{x^2+y^2}+\frac{1}{2}\ln{(x^{2}+y^{2})}.
    \end{eqnarray*}

    Since $H(x,-1)=H(x,1)$, if $(x,-1)$ lies on a level curve of $H$, then $(x,1)$ also lies on same level curve. 
    Thus, if $z(t)$ is a solution of the equation $\dot{z}=\frac{iz^{2}}{1+z}$ with initial contion $z(0)=x-i$, 
    then there exists a time $t_{0}$ such that $z(t_{0})=x+i$. 

    Consider $w_{0}\in \Sigma_{2}$, we have
    \begin{eqnarray*}
        z^{-}(t)=(w_{0}+i-x_{0})e^{(c+id)t}-(i-x_{0}),
    \end{eqnarray*}
    is a solution of $\dot{z}=(c+id)(z+i-z_{0})$ with $z^{-}(0)=w_{0}$ and
    \begin{eqnarray*}
        z^{-}\left( \frac{\pi }{d}\right)=-(s-x_{0})e^{\frac{c\pi }{d}}+x_{0}-i.
    \end{eqnarray*}

    Let $z(t)$ be a solution of the equation $\dot{z}=\frac{iz^{2}}{1+z}$, with the initial condition 
    $z(0)=-(s-x_{0})e^{\frac{c\pi }{d}}+x_{0}-i$, then there exists $t_{0}$ such tach $z(t_{0})=-(s-x_{0})e^{\frac{c\pi }{d}}+x_{0}+i$.

    We have that 
    \begin{eqnarray*}
        z^{+}(t)=(-(s-x_{0})e^{\frac{c\pi }{d}}+x_{0}+i-i+x_{1})e^{(a+ib)t}-(-i+x_{1})
    \end{eqnarray*}
    is solution of $\dot{z}=(a+ib)(z-i+x_{1})$,such that $z^{+}(t)=-(s-x_{0})e^{\frac{c\pi }{d}}-x_{0}+i$ and
    \begin{eqnarray*}
        z^{+}\left( \frac{\pi }{a}\right)=((s-x_{0})e^{\frac{c\pi }{d}}-x_{0}-x_{1})e^{\frac{a\pi }{b}}-x_{1}+i.
    \end{eqnarray*}

    At last, let $\tilde{z}(t)$ be a solution of the equation $\dot{z}=\frac{iz^{2}}{1+z}$ with
     $\tilde{z}(0)=((s-x_{0})e^{\frac{c\pi }{d}}-x_{0}-x_{1})e^{\frac{a\pi }{b}}-x_{1}+i$, then 
     there exists $t_{1}$ such that $\tilde{z}(t_{1})=((s-x_{0})e^{\frac{c\pi }{d}}-x_{0}-x_{1})e^{\frac{a\pi }{b}}-x_{1}-i$.

    Therefore, the Poincaré map around $w_{0} = s_{0} + i$ is given by
    \begin{eqnarray*}
        \Pi(s - i) = \left((s - x_{0})e^{\frac{c\pi}{d}} - x_{0} - x_{1}\right)e^{\frac{a\pi}{b}} - x_{1} - i
    \end{eqnarray*}

    To conclude, we seek the solutions of the equation $ \Pi(s - i) = s - i$. The number of 
    solutions to this equation corresponds to the number of limit cycles of the system.

    Since  $\frac{a}{b} + \frac{c}{d} \neq 0$, the equation  $\Pi(s - i) = s - i$  has a unique solution, 
    which is $s = \frac{-x_{0}e^{\frac{c\pi}{d} + \frac{a\pi}{b}} - (x_{0} + x_{1})e^{\frac{a\pi}{b}} - x_{1}}{1 - e^{\frac{c\pi}{d} + 
    \frac{a\pi}{b}}}$, then the system (\ref{sistema2}) has a limit cycle.
    
\end{proof}

\begin{cor}
    Under the hypotheses of proposition (\ref{lc2}), we have that the systems
\begin{eqnarray}\label{sistema2.1}
\begin{cases}
\dot{w}=\left(\frac{b-ai}{2}\right)(w-1)(-2i+(w-1)(-i+x_{1})), \quad \hbox{if} \quad w\in \Sigma_{+}^{E},\\
\dot{w}=\frac{2(w-1)}{w-1-2i}, \quad \quad \hbox{if} \quad w\in \Sigma_{c}^{E},\\
\dot{w}=\left(\frac{d-ci}{2}\right)(w-1)(-2i+(w-1)(i-x_{0}) ), \quad \hbox{if} \quad w\in \Sigma_{-}^{E},
\end{cases}
\end{eqnarray}
and
\begin{eqnarray}\label{sistema2.2}
\begin{cases}
\dot{w}=\left(\frac{b-ai}{2}\right)(w-1)(-2iw+(w-1)(-i+x_{1})), \quad \hbox{if} \quad w\in \Sigma_{+}^{I},\\
\dot{w}=\frac{2w^{2}(w-1)}{w(-2i+1)-1}, \quad \quad \hbox{if} \quad w\in \Sigma_{c}^{E},\\
\dot{w}=\left(\frac{d-ci}{2}\right)(w-1)(-2iw+(w-1)(i-x_{0}) ), \quad \hbox{if} \quad w\in \Sigma_{-}^{I},
\end{cases}
\end{eqnarray}
  has a limit cycle.  
\end{cor}

Observe that the conditions stated in the previous proposition are nontrivial. For instance, 
by choosing $a=c=-1$, $b=d=1$, $x_{0}=1$ and $x_{1}=2$, one can verify the existence of limit cycle, as show in figure (\ref{img3.2}).

\begin{figure}[!htb]
\centering
\includegraphics[scale=0.56]{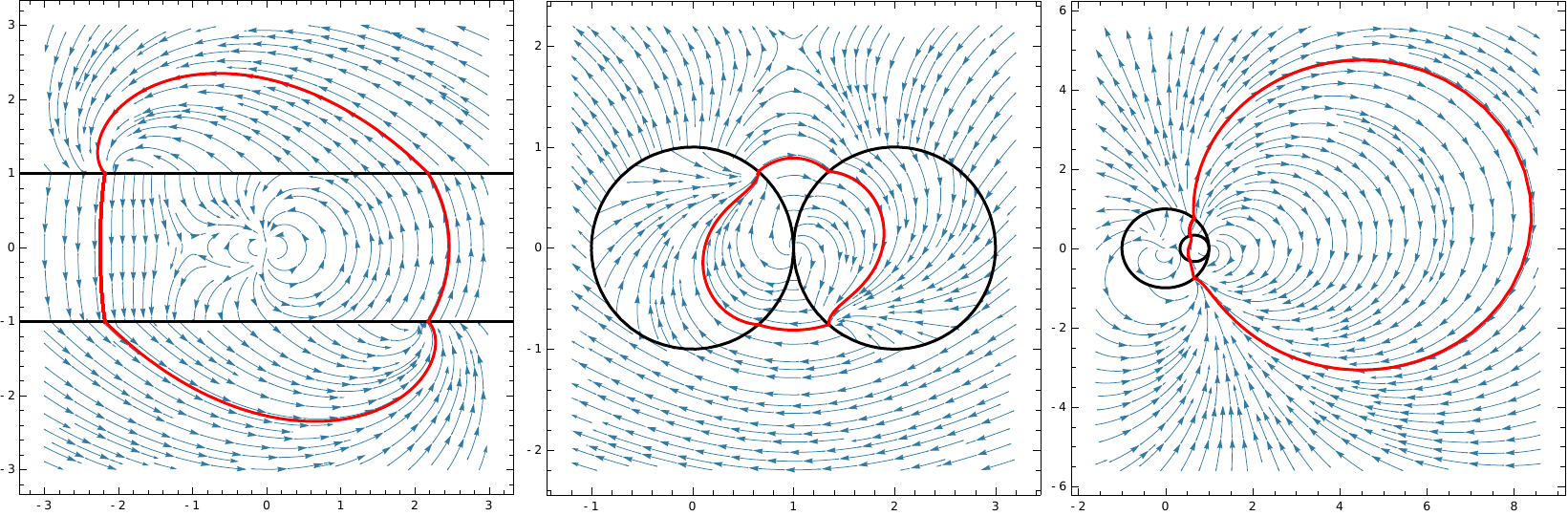}
\caption{Phase portrait of system (\ref{sistema2}), (\ref{sistema2.1}) and  (\ref{sistema2.2}) with
 $a=c=-1$, $b=d=1$ and $x_{0}=x_{1}=2$ .}
\label{img3.2}
\end{figure}

\begin{pro}\label{lc5}
    Let $a,b,c,d,x_{0}$ and $x_{1}$ be non-zero real number such that $a,c<0$ e $b,d>0$, $x_{0},x_{1}>1$. Then the following  system 
\begin{eqnarray}\label{sistema5}
\begin{cases}
\dot{z}=(a+ib)(z-i+x_{1}), \quad \hbox{if} \quad z\in \Sigma_{+},\\
\dot{z}=\dfrac{i}{z}, \quad \quad \hbox{if} \quad z\in \Sigma_{c},\\
\dot{z}=(c+id)(z+i-x_{0}), \quad \hbox{if} \quad z\in \Sigma_{-},
\end{cases}
\end{eqnarray}
 has a limit cycle.
\end{pro}

\begin{proof}
Consider the equation $\dot{z}=f(z)=\frac{i}{z}$. Initially, observe that the solutions to this equation are symmetric with respect to the $x$-axis. In fact, in polar coordinates the equation $\dot{z}=f(z)$ can be rewritten as
\begin{eqnarray*}
\begin{cases}
\dot{r}=r^{-1}\sin 2\theta, \\
\dot{\theta}=r^{-2}\cos 2\theta .
\end{cases}
\end{eqnarray*}

 Thus, it follows that the orbits of the equation satisfy
\begin{eqnarray*}
 r=\dfrac{k}{|\cos 2\theta |^{\frac{1}{2}}}.
\end{eqnarray*}

Since $r(-\theta )=r(\theta)$, it follows that the orbits of the equation $\dot{z}=f(z)$ are symmetry with respect to the $x$-axis.

Let
\begin{eqnarray*}
    z^{-}(t)=(w_{0}+i-x_{0})e^{(c+id)t}-(i-x_{0}),
\end{eqnarray*}
be the solution of equation $\dot{z}=(c+id)(z+i-x_{0})$ with the initial condition $z^{-}(0)=w_{0}=s-i$. Note that 
\begin{eqnarray*}
    z^{-}\left( \frac{\pi }{d}\right)=-(s-x_{0})e^{\frac{c\pi }{d}}+x_{0}-i.
\end{eqnarray*}

Consider $z(t)$ the solution of the equation $\dot{z}=f(z)$ with the initial condition $z(0)=(s-x_{0})
e^{\frac{c\pi }{d}}+x_{0}-i$. Due to the symmetry of the solutions of $\dot{z}=f(z)$, there exists a 
$t_{0}$ such that $z(t_{0})=(s-x_{0})e^{\frac{c\pi }{d}}+x_{0}+i$.

Now, let 
\begin{eqnarray*}
    z^{+}(t)=(-(s-x_{0})e^{\frac{c\pi }{d}}+x_{0}+x_{1})e^{(a+ib)t}-(-i+x_{1}),
\end{eqnarray*}
be the solution of the equation $\dot{z}=(a+ib)(z+i-x_{1})$ with the initial condition 
$z^{+}(0)=z^{+}(0)=-(s-x_{0})e^{\frac{c\pi }{d}}-x_{0}+i$ and $z^{+}\left( \frac{\pi }{a}\right)=
((s-x_{0})e^{\frac{c\pi }{d}}-x_{0}-x_{1})e^{\frac{a\pi }{b}}-x_{1}+i$.

Finally, consider $\tilde{z}(t)$ the solution of the equation $\dot{z}=f(z)$ with the initial condition 
$\tilde{z}(0)=((s-x_{0})e^{\frac{c\pi }{d}}-x_{0}-x_{1})e^{\frac{a\pi }{b}}-x_{1}+i$. Due to the 
symmetry of the solution of $\dot{z}=f(z)$, there exists a $t_{1}$ such that $\tilde{z}(t_{1})=
((s-x_{0})e^{\frac{c\pi }{d}}-x_{0}-x_{1})e^{\frac{a\pi }{b}}-x_{1}-i$. Thus, it follows that the 
Poincaré map around $w_{0}=s-i$ is given by
\begin{eqnarray*}
    \Pi(s-i)=((s-x_{0})e^{\frac{c\pi }{d}}-x_{0}-x_{1})e^{\frac{a\pi }{b}}-x_{1}-i.
\end{eqnarray*}

Since the equation $\Pi(w_{0})=w_{0}$ has a unique solution, given by
 \begin{eqnarray*}
    w_{0}=\dfrac{-x_{0}e^{\frac{c\pi }{d}+\frac{a\pi }{b}}-(x_{0}+x_{1})e^{\frac{a\pi }{b}})-x_{1}}{(1-e^{\frac{c\pi }{d}+\frac{a\pi }{b}})}-i.    
\end{eqnarray*}
we can conclude that the system has a limit cycle. 
\end{proof}

Although the previous proposition was stated for $\dot{z}=\frac{i}{z}$, it can be extended to 
the case $\dot{z}=\frac{i}{z^{n}}$, with $n$ odd. To do so, the values of $x_{0}$ and $x_{1}$ must 
be adjusted according to the value of $n$. In fact, just as in the case $\dot{z}=\frac{i}{z}$, the equation $\dot{z}=\frac{i}{z^{n}}$ also exhibits symmetry with respect to  the real axis. The arguments used in the proof remain valid in this more general setting.
\begin{cor}
    Under the hypotheses of proposition (\ref{lc3}), we have that the systems
\begin{eqnarray}\label{sistema5.1}
\begin{cases}
\dot{w}=\left(\frac{b-ai}{2}\right)(w-1)(-2i+(w-1)(-i+x_{1})), \quad \hbox{if}\quad w\in \Sigma_{+}^{E},\\
\dot{w}=\frac{i}{4}(w-1)^{3}, \quad \quad \hbox{if} \quad w\in \Sigma_{c}^{E},\\
\dot{w}=\left(\frac{d-ci}{2}\right)(w-1)(-2i+(w-1)(i-x_{0}) ), \quad \hbox{if}\quad w\in \Sigma_{-}^{E},
\end{cases}
\end{eqnarray}
and
\begin{eqnarray}\label{sistema5.2}
\begin{cases}
\dot{w}=\left(\frac{b-ai}{2}\right)(w-1)(-2iw+(w-1)(-i+x_{1})), \quad \hbox{if} \quad w\in \Sigma_{+}^{I},\\
\dot{w}=\frac{-i(w-1)^{3}}{4w}, \quad \quad \hbox{if} \quad w\in \Sigma_{c}^{I},\\
\dot{w}=\left(\frac{d-ci}{2}\right)(w-1)(-2iw+(w-1)(i-x_{0}) ), \quad \hbox{if}\quad w\in \Sigma_{-}^{I},
\end{cases}
\end{eqnarray}
  has a limit cycle.  
\end{cor}
It is important to note that the conditions provided in the previous proposition are not trivial. 
In fact, by taking $n=1$, $a=c=-1$, $b=d=1$ and $x_{0}=x_{1}=1$, we obtain the existence 
of a limit cycle (as illustrated in figure \ref{img1.sela}).
\begin{figure}[!htb]\label{img3.5}
\centering
\includegraphics[scale=0.56]{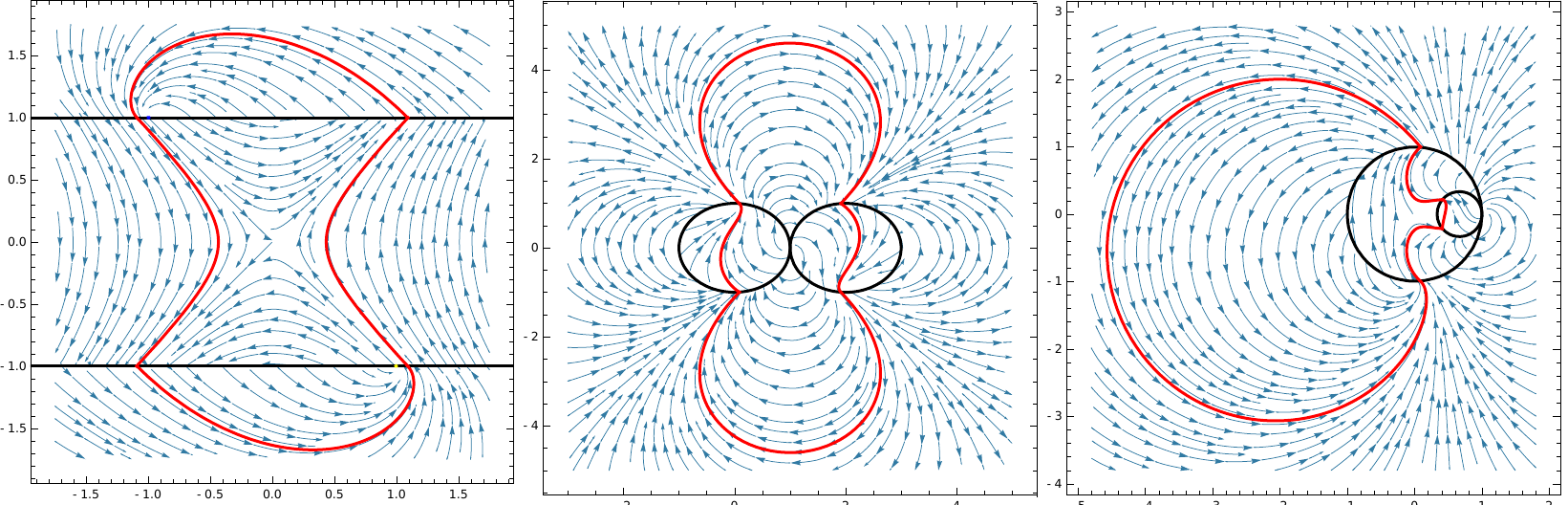}
\caption{Phase portrait of systems (\ref{sistema5}), (\ref{sistema5.1}) and (\ref{sistema5.2}) with
 $a=c=-1$, $b=d=1$ and $x_{0}=x_{1}=1$ .}
\label{img1.sela}
\end{figure}

\section{Number of limit cycles that bifurcate from linear centers}
In this section, we will present the proofs of Theorems \ref{teoa}, \ref{teob} and \ref{teoc}. However, before proceeding, it is important to note that  if $f=u+iv$ is a continuous function and $\gamma $ is a smooth curve, then 
\begin{eqnarray*}
\int_{\gamma }f(z)dz=\int_{\gamma }udx-vdy+i\int_{\gamma }udy+vdx.
\end{eqnarray*}

Hence, if we consider $i\overline{f}$, then
\begin{eqnarray*}
i\overline{f}=i(u-iv)=v+iu
\end{eqnarray*}
and 
\begin{eqnarray*}
\int_{\gamma }i\overline{f}dz=\int_{\gamma } vdx-udy+i\int_{\gamma }vdy+udx.
\end{eqnarray*}

Therefore if
\begin{equation}\label{eq10}
\dot{z}=A(z)+\varepsilon B^{\sigma}(z), \quad z\in\Sigma_{\sigma}, \quad \sigma=+,c,-\end{equation}
where $A(x+iy)=u(x,y)+iv(x,y)=H_{x}-iH_{y} $ and the unperturbed system satisfy the Assumptions $1$ and $2$ then we have:
\begin{lema}\label{melnikov- sh}
    Consider system $(\ref{eq10})$ with $\varepsilon $ sufficiently small, then the first Melnikov function is
\[
    M(h)=\Re\left(\int_{L_{h}^{-}}i\overline{B^{-}}(z)dz \right)+\Re\left(\int_{L_{h}^{C_{1}}}i\overline{B^{c}}(z) dz\right)+\]\[+
    \Re\left(\int_{L_{h}^{+}}i\overline{B^{+}}(z) dz\right)+\Re\left(\int_{L_{h}^{C_{2}}}i\overline{B^{c}}(z) dz\right).
\]  
\end{lema}
\refstepcounter{remark}
\noindent\textbf{Remark  \theremark.}

   Note that this result can be be extended to the case where any one of the functions $B^{+}$, $B^{C}$, or $B^{-}$ is rational. To fix ideas, suppose that  $B^{+}(z)$ is a rational function, that is, $B^{+}(z)=\frac{p(z)}{q(z)}$, where $p$ and $q$ are polynomials. Furthermore, suppose that the zeros of $q$ do not lie on any periodic orbit of the family $L_{h}$. In this case, we can still apply the lemma.

    Indeed, note that
    \begin{eqnarray*}
        \dot{z}=A(z)+\varepsilon\dfrac{p(z)}{q(z)}=\dfrac{A(z)|q(z)|^{2}+\varepsilon p(z)\overline{q}(z)}{|q(z)|^{2}}.
    \end{eqnarray*}

    Since $|q(z)|^{2}$ is positive, the equation has the same phase portrait as the equation
    \begin{eqnarray}\label{nlb.o1}
       \dot{z}= A(z)|q(z)|^{2}+\varepsilon p(z)\overline{q}(z)
    \end{eqnarray}

    Note that, for $\varepsilon=0$, the system (\ref{nlb.o1}) is Hamiltonian and $R^{+}(z)=\frac{1}{|q(z)|^{2}}$ is an integrating factor. Then
 \[
\begin{aligned}
M(h) &= \Re\left(\int_{L_{h}^{-}} i\, \overline{B^{-}}(z)\, dz \right) 
+ \Re\left(\int_{L_{h}^{C_{1}}} i\, \overline{B^{c}}(z)\, dz \right) \\
&\quad + \Re\left(\int_{L_{h}^{+}} i\, \frac{\overline{p(z)}\, q(z)}{q(z)\, \overline{q(z)}}\, dz \right)
+ \Re\left(\int_{L_{h}^{C_{2}}} i\, \overline{B^{c}}(z)\, dz \right) \\
&= \Re\left(\int_{L_{h}^{-}} i\, \overline{B^{-}}(z)\, dz \right) 
+ \Re\left(\int_{L_{h}^{C_{1}}} i\, \overline{B^{c}}(z)\, dz \right) \\
&\quad + \Re\left(\int_{L_{h}^{+}} i\, \overline{B^{+}}(z)\, dz \right) 
+ \Re\left(\int_{L_{h}^{C_{2}}} i\, \overline{B^{c}}(z)\, dz \right).
\end{aligned}
\]

\renewcommand{\proofname}{Proof of Theorem \ref{teoa}}
\begin{proof}
Initially, note that for  $\varepsilon =0$ the system (\ref{eq.it1}) is Hamiltonian with
\begin{eqnarray*}
H^{+}(x,y)=H^{c}(x,y)=H^{+}(x,y)=\frac{x^{2}+y^{2}}{2}.
\end{eqnarray*}

Also note that the system (\ref{eq.it1}) has a family of periodic orbits $L_{h}=L^{-}_{h}\cup L_{h}^{c_{1}}\cup  L_{h}^{+}\cup L_{h}^{c_{2}}$ for $h\in (\frac{1}{2},\infty )$, where
\begin{eqnarray*}
& &L^{-}_{h}=\left\{ \sqrt{2h}e^{it}, \pi +\arcsin \left( \frac{1}{\sqrt{2h}}\right) \leq t\leq 2\pi-\arcsin\left( \dfrac{1}{\sqrt{2h}}\right)\right\} ,\\
& &L^{c_{1}}_{h}=\left\{ \sqrt{2h}e^{it}, -\arcsin \left( \frac{1}{\sqrt{2h}}\right) \leq t\leq \arcsin\left( \dfrac{1}{\sqrt{2h}}\right)\right\} ,\\
& &L^{+}_{h}=\left\{  \sqrt{2h}e^{it}, \arcsin \left( \frac{1}{\sqrt{2h}}\right) \leq t\leq \pi - \arcsin\left( \dfrac{1}{\sqrt{2h}}\right)\right\} ,\\
& &L^{c_{2}}_{h}=\left\{  \sqrt{2h}e^{it}, \pi-\arcsin \left( \frac{1}{\sqrt{2h}}\right) \leq t\leq \pi + \arcsin\left( \dfrac{1}{\sqrt{2h}}\right)\right\} .\\
\end{eqnarray*}

Furthermore, each one of periodic orbits $L_{h}$ intersects the straight line $\Im(z)=-1$ at two points $A^{-}(h)=\left( -\sqrt{2h-1},-1\right)$ and $B^{-}(h)=\left( \sqrt{2h-1},-1\right)$, and the straight line $\Im(z)=1$ at the points $A^{+}(h)=\left( \sqrt{2h-1},-1\right)$ and $B^{+}(h)=\left( -\sqrt{2h-1},1\right)$.

So the Assumptions $1$ and $2$ are satisfied. By Theorem \ref{melnikov}, to estimate the number of limit cycles that bifurcate from the center $\dot{z}=iz$, we need to estimate the number of zeros of the Melnikov function.

As $H^{-}(x,y)=H^{c}(x,y)=H^{+}(x,y)$, then by Lemma  \ref{melnikov- sh}, the expression for the Melnikov function is
\[
M(h)=\Re \left(\int_{L_{h}^{-}}i\overline{h^{-}}dz\right)+\Re \left(\int_{L_{h}^{c_{1}}}i\overline{h^{c}}dz\right)+\]\[+\Re \left(\int_{L_{h}^{+}}i\overline{h^{+}}dz\right)+\Re \left(\int_{L_{h}^{c_{2}}}i\overline{h^{c}}dz\right).
\]
Each integral will be handled independently. As a preliminary step, note that
\begin{eqnarray*}
i\overline{h^{-}(z)}=i\sum_{j=0}^{\ell } (a_{j}^{-}-ib_{j}^{-})\overline{z}^{j}=\sum_{j=0}^{\ell }(b_{j}^{-}+ia_{j}^{-})\overline{z}^{j}.
\end{eqnarray*}

For convenience, set $r=\sqrt{2h}$ and $\theta_{r}=\arcsin \left( \frac{1}{r}\right)$. Using the definitions, we have 

\begin{eqnarray*}
\int_{L^{-}_{h}}i\overline{h^{-}}dz&=&\int_{L_{h}^{-}}\sum_{j=0}^{\ell }(b_{j}^{-}+ia_{j}^{-})\overline{z}^{j}dz\\ &=&
\int_{\pi +\theta_{r}}^{2\pi -\theta_{r}}\left(\sum_{j=0}^{\ell }(b_{j}^{-}+ia_{j}^{-})r^{j}\overline{e^{ijt}} \right)ire^{it}dt\\
&=&\int_{\pi +\theta_{r}}^{2\pi -\theta_{r}}\left(\sum_{j=0}^{\ell }(-a_{j}^{-}+ib_{j}^{-})r^{j+1}e^{-i(j-1)t}\right)dt
\end{eqnarray*}

We define the following functions
\begin{eqnarray*}
    A_{j}^{-}(t )&=&-a_{j}^{-}\cos((j-1)t)+b_{j}^{-}\sin((j-1)t),\\
    B_{j}^{-}(t )&=&a_{j}^{-}\sin((j-1)t)+b_{j}^{-}\cos((j-1)t),\\
    C_{j}^{-}(t )&=&-a_{j}^{-}\sin((j-1)t)-b_{j}^{-}\cos((j-1)t).
\end{eqnarray*}

With this notation, we have 
\begin{eqnarray*}
    (-a_{j}^{-}+ib_{j}^{-})e^{-i(j-1)t}=A_{j}^{-}(t )+iB_{j}^{-}(t ),
\end{eqnarray*}
as a result

\begin{align*}
\int_{L_{h}^{-}} i\, \overline{h^{-}}\, dz 
&= \int_{\pi + \theta_{r}}^{2\pi - \theta_{r}} 
\left( -a_{0}^{-} r \cos t - b_{0}^{-} r \sin t - a_{1}^{-} r^{2} 
+ \sum_{j=2}^{\ell } \left( r^{j+1}A_{j}^{-}(t )\right) \right) dt \\
&\quad + i \int_{\pi + \theta_{r}}^{2\pi - \theta_{r}} 
\left( b_{0}^{-} r \cos t - a_{0}^{-} r \sin t + b_{1}^{-} r^{2} 
+ \sum_{j=2}^{\ell } \left( r^{j+1} B_{j}^{-}(t ) \right) \right) dt
\end{align*}

Therefore,
\begin{eqnarray*}
\Re \left( \int_{L_{h}^{-}}i\overline{h^{-}}dz\right)&=&\int_{\pi +\theta_{r}}^{2\pi -\theta_{r}}\left( -a_{0}^{-}r\cos t-b_{0}^{-}r\sin t -a_{1}^{-}r^{2}+\sum_{j=2}^{\ell }r^{j+1}A_{j}^{-}(t )\right) dt\\
&=&\bigg[-a_{0}^{-}r\sin t+b_{0}^{-}r\cot t-a_{1}^{-}r^{2}t+\sum_{j=2}^{\ell }\dfrac{r^{j+1}}{j-1}C_{j}^{-}(t )\bigg]_{\pi +\theta_{r} }^{2\pi -\theta_{r}}.
\end{eqnarray*}

Given that $\sin (\ell (2\pi -x))=-\sin(\ell x)$, $\cos (\ell (2\pi-x))=\cos (\ell x)$, $\sin (\ell (\pi+x))=(-1)^{\ell }\sin (\ell x)$ and $\cos (\ell(\pi+x))=(-1)^{\ell }\cos (\ell x)$, then

\begin{eqnarray}
\Re \left( \int_{L_{h}^{-}}i\overline{h^{+}}dz\right)&=&2b_{0}^{-}r\cos(\theta_{r})-a_{1}^{-}\pi r^{2}+2a_{1}^{-}r^{2}\theta_{r} 
+2\sum_{j=1}^{\lfloor \frac{ \ell -1}{2}\rfloor}\dfrac{a_{2j+1}^{-}r^{2j+2}}{2j}\sin (2j\theta_{r}) \notag \\&-&2\sum_{j=1}^{\lfloor \frac{\ell }{2}\rfloor }\dfrac{b_{2j}^{-}r^{2j+1}}{2j-1}\cos ((2j-1)\theta_{r}).\label{teoA.eq2}
\end{eqnarray}

Using that, $\sin (-\ell x)=-\sin (\ell x)$ and $\cos (-\ell x)=\cos (\ell x)$, we get
\begin{eqnarray}\label{teoA.eq3}
\Re \left(\int_{L_{h}^{c_{1}}}i\overline{h^{c}}dz\right)=-2a_{0}^{c}r\sin \theta_{r}-2a_{1}^{c}r^{2}\theta_{r}-2\sum_{j=2}^{\ell }\dfrac{a_{j}^{c}r^{j+1}}{j-1}\sin \left((j-1)\theta_{r}\right).
\end{eqnarray}

Nothing that  $\sin (\ell (\pi -x))=(-1)^{\ell +1}\sin (\ell x)$ and $\cos (\ell (\pi -x))=(-1)^{\ell }\cos (\ell x)$, we have
\begin{eqnarray}\label{teoA.eq4}
\Re\left(\int_{L_{h}^{+}}i\overline{h^{+}}dz\right)&=&-2b_{0}^{+}r\cos \theta_{r}-a_{1}^{+}\pi r^{2}+2a_{1}^{+}r^{2}\theta_{r}
+2\sum_{j=1}^{\lfloor\frac{\ell-1}{2}\rfloor}\dfrac{a^{+}_{2j+1}r^{2j+2}}{2j}\sin (2j\theta_{r})\notag \\&+&2\sum_{j=1}^{\lfloor\frac{\ell }{2}\rfloor}\dfrac{b^{+}_{2j}r^{2j+1}}{2j-1}\cos ((2j-1)\theta_{r}).
\end{eqnarray}

Finally, considering that $\sin \left(\ell (\pi -x)\right)=(-1)^{\ell +1}\sin \left (\ell \pi \right)$, $\sin\left( \ell (\pi +x)\right)=(-1)^{\ell }\sin \left( \ell x\right) $ and $\cos \left( \ell (\pi -x\right) =(-1)^{\ell }\cos \left( \ell x\right) =\cos \left( \ell (\pi +x)\right) $, we can conclude that
\begin{eqnarray}\label{teoA.eq5}
\Re \left( \int_{L_{h}^{c^{2}}}i\overline{h^{c}}dz\right) =2a_{0}^{c}r\sin \theta_{r}-2a_{1} r^{2}\theta_{r}+2\sum_{j=2}^{\ell }\dfrac{a_{j}^{c}r^{j+1}}{j-1}(-1)^{j}\sin \left( (j-1)\theta_{r}\right).
\end{eqnarray} 

Based on  (\ref{teoA.eq2}), (\ref{teoA.eq3}), (\ref{teoA.eq4}) and (\ref{teoA.eq5}), the expression of Melnikov function follows

\begin{eqnarray*}
M(r)&=&(2b_{0}^{-}-2b_{0}^{+})r\cos (\theta_{r})+(-\pi a_{1}^{-}-\pi a_{1}^{+})r^{2}+(2a_{1}^{-}-4a_{1}^{c}+2a_{1}^{+})r^{2}\theta_{r}\\ 
&+& \sum_{j=1}^{\lfloor \frac{\ell }{2} \rfloor }\dfrac{(-2b_{2j}^{-}+2b_{2j}^{+})}{2j-1}r^{2j+1}\cos ((2j-1)\theta_{r}))\\
&+&\sum_{j=1}^{\lfloor \frac{\ell -1}{2}\rfloor}\dfrac{(2a_{2j+1}^{-}+2a_{2j+1}^{+}-4a_{2j+1}^{c})}{2j}r^{2j+2}\sin (2j\theta_{r}).
\end{eqnarray*}

In particular, when $\ell =4$, we have
\begin{eqnarray*}
    M(r)=A_{1}f_{1}+A_{2}f_{2}+A_{3}f_{3}+A_{4}f_{4}+A_{5}f_{5}+A_{6}f_{6},
\end{eqnarray*}
where
\begin{eqnarray*}
    f_{1} &=& r\cos \theta_{r}, \quad f_{2} = r^{2}, \quad f_{3} = r^{2}\theta_{r}, \\
    f_{4} &=& r^{3}\cos \theta_{r}, \quad f_{5} = r^{5}\cos 3\theta_{r}, \quad f_{6} = r^{4}\sin 2\theta_{r}
\end{eqnarray*}
and
\begin{eqnarray*}
   A_{1} &=& 2b_{0}^{-} - 2b_{0}^{+}, \quad A_{2} = -\pi a_{1}^{-} - \pi a_{1}^{+}, \quad A_{3} = 2a_{1}^{-} - 4a_{1}^{c} + 2a_{1}^{+} \\
    A_{4} &=& -2b_{2}^{-} + 2b_{2}^{+}, \quad A_{5} = \frac{-2b_{4}^{-} + 2b_{4}^{+}}{2}, \quad A_{6} = \frac{2a_{3}^{-} + 2a_{3}^{+} - 4a_{3}^{c}}{2}.
\end{eqnarray*}
Note that $f_{6}=2r^{3}\cos \theta_{r}$, and it follows that $f_{6}=2f_{4}$. Therefore, $M(r)$ can be rewritten as
\begin{eqnarray*}
    M(r)=\alpha_{1}f_{1}+\alpha_{2}f_{2}+\alpha_{3}f_{3}+\alpha_{4}f_{4}+\alpha_{5}f_{5},
\end{eqnarray*}
with $\alpha_{1}=A_{1}$, $\alpha_{2}=A_{2}$, $\alpha_{3}=A_{3}$, $\alpha_{4}=A_{4}+2A_{6}$ and $\alpha_{5}=A_{5}$.

Since $W(f_{1},f_{2},f_{3},f_{4},f_{5})(2)=8$, this guarantees that the functions $f_{1},f_{2},f_{3},f_{4}$ and $f_{5}$ are linearly independent. Moreover, $f_{2} $ has constant sign on the interval $\left(\frac{1}{2},\infty \right)$. Given that the coefficients of the functions  $f_{1},f_{2},f_{3},f_{4}$ and $f_{5}$ can be chosen arbitrarily, Lemma \ref{lem1}, there exist constant that ensure $M(r)$ has $4$ simple zeros.

Therefore, system (\ref{eq.it1}) has at least $4$ limit cycles bifurcating from the periodic orbits of the center $\dot{z}=iz$.
\end{proof}

\renewcommand{\proofname}{Proof of Theorem \ref{teob}}
\begin{proof}
By making $w=-\frac{2i}{z-1}$, the system (\ref{eq.it2}) is transformed into

\begin{equation}\label{eq.sec3.1}
\dot{w}=iw+\varepsilon \tilde{h}^{\sigma}(w),  \quad w\in\Sigma_{\sigma},\quad \sigma=+,c,-,
\end{equation}
where 
\begin{equation}
        \tilde{h}^{\sigma}(w) = \sum_{j=0}^{\ell } \left( \frac{b_{j}^{\sigma} - i a_{j}^{\sigma}}{2} \right) (w - 2i)^{j} w^{2 - j}.      
\end{equation}

Next, we will manipulate the original sums to rewrite the expressions for $h^{+},h^{c}$ and $h^{-}$. By Newton's binomial theorem, we have that
\begin{eqnarray*}
    (w-2i)^{j}=\sum_{k=0}^{j}\binom{j}{k}w^{j-k}(-2i)^{k},
\end{eqnarray*}
therefore
\begin{eqnarray*}
    (w-2i)^{j}w^{2-j}=\sum_{k=0}^{j}\binom{j}{k}(-2i)^{k}w^{2-k},
\end{eqnarray*}
and thus 
\begin{eqnarray*}
   \tilde{h}^{+}(w)=\sum_{j=0}^{\ell }\left( \sum_{k=j}^{\ell }\left(\frac{b_{j}^{+} - i a_{j}^{+}}{2}\right)\binom{j}{k}(-2i)^{k}\right)w^{2-k}.
\end{eqnarray*}

Thus, we have that 
\begin{eqnarray*}
  \tilde{h}^{+}(w)=  \sum_{j=0}^{\ell }(c_{2-j}^{+}+id_{2-j}^{+})w^{2-j},\quad  
\end{eqnarray*}
where 
\begin{eqnarray*}
    (c_{2-j}^{+}+id_{2-j}^{+})=\left( (-2i)^{j}\sum_{k=j}^{\ell }\left(\frac{b_{k}^{+} - i a_{k}^{+}}{2}\right)\binom{k}{j}\right).
\end{eqnarray*}

Analogously to the obtaining of the expression for $h^{+}(w)$, the expressions for $\tilde{h}^{c}(w)$ and $\tilde{h}^{-}(w)$ can be obtained using the same procedure, with the substitution of the coefficients $a_{j}^{+},b_{j}^{+}$ by $a_{j}^{c},b_{j}^{c}$ and $a_{j}^{-},b_{j}^{-}$, respectively.

For  $\varepsilon =0$ the system (\ref{eq.sec3.1}) is Hamiltonian with
\begin{eqnarray*}
H^{+}(x,y)=H^{c}(x,y)=H^{+}(x,y)=\frac{x^{2}+y^{2}}{2}.
\end{eqnarray*}

In addition, the system (\Ref{eq.sec3.1}) has a family of periodic orbits $L_{h}=L^{-}_{h}\cup L_{h}^{c_{1}}\cup  L_{h}^{+}\cup L_{h}^{c_{2}}$ for $h\in (\frac{1}{2},\infty )$, where
\begin{eqnarray*}
& &L^{-}_{h}=\left\{  \sqrt{2h}e^{it}, \pi +\arcsin \left( \frac{1}{\sqrt{2h}}\right) \leq t\leq 2\pi-\arcsin\left( \dfrac{1}{\sqrt{2h}}\right)\right\} ,\\
& &L^{c_{1}}_{h}=\left\{ \sqrt{2h}e^{it}, -\arcsin \left( \frac{1}{\sqrt{2h}}\right) \leq t\leq \arcsin\left( \dfrac{1}{\sqrt{2h}}\right)\right\} ,\\
& &L^{+}_{h}=\left\{ \sqrt{2h}e^{it}, \arcsin \left( \frac{1}{\sqrt{2h}}\right) \leq t\leq \pi - \arcsin\left( \dfrac{1}{\sqrt{2h}}\right)\right\} ,\\
& &L^{c_{2}}_{h}=\left\{ \sqrt{2h}e^{it}, \pi-\arcsin \left( \frac{1}{\sqrt{2h}}\right) \leq t\leq \pi + \arcsin\left( \dfrac{1}{\sqrt{2h}}\right)\right\} .\\
\end{eqnarray*}

Furthermore, each one of periodic orbits $L_{h}$ intersects the straight line $\Im(w)=-1$ at two points $A^{-}(h)=\left( -\sqrt{2h-1},-1\right)$ and $B^{-}(h)=\left( \sqrt{2h-1},-1\right)$, and the straight line $\Im(w)=1$ at the points $A^{+}(h)=\left( \sqrt{2h-1},-1\right)$ and $B^{+}(h)=\left( -\sqrt{2h-1},1\right)$.

So the Assumptions $1$ and $2$ are satisfied. By Theorem \ref{melnikov}, to estimate the number of limit cycles that bifurcate from the center $\dot{z}=iz$, we need to estimate the number of zeros of the Melnikov function.

As $H^{-}(x,y)=H^{c}(x,y)=H^{c}(x,y)$, then by Lemma  \ref{melnikov- sh}, the expression for the Melnikov function is
\begin{eqnarray*}
M(h)=\Re \left(\int_{L_{h}^{-}}i\overline{\tilde{h}^{-}}dz\right)+\Re \left(\int_{L_{h}^{c_{1}}}i\overline{\tilde{h}^{c}}dz\right)+\Re \left(\int_{L_{h}^{+}}i\overline{\tilde{h}^{+}}dz\right)+\Re \left(\int_{L_{h}^{c_{2}}}i\overline{\tilde{h}^{c}}dz\right).
\end{eqnarray*}

We will calculate each of the integrals separately, but first, note that
\begin{eqnarray*}
    i\overline{\tilde{h}^{-}(z)}=\sum_{j=0}^{\ell }(d_{2-j}^{-}+ic_{2-j}^{-})\overline{w}^{2-j}.
\end{eqnarray*}

To simplify the writing, we will set  $r =\sqrt{2h}$ and  $\theta_{r}=\arcsin \left( \frac{1}{r}\right)$, we have that
\begin{eqnarray*}
    \int_{L^{-}_{h}}i\overline{\tilde{h}^{-}}dw&=&\int_{L_{h}^{-}}\sum_{j=0}^{\ell }(d_{2-j}^{-}+ic_{2-j}^{-})\overline{w}^{2-j}dw\\
    &=&\int_{\pi+\theta_{r}}^{2\pi -\theta_{r}}\left( \sum_{j=0}^{\ell }(d_{2-j}^{-}+ic_{2-j}^{-})r^{2-j}e^{-i(2-j)t}\right)ire^{it}dt\\
    &=&\int_{\pi+\theta_{h}}^{2\pi -\theta_{r}}\left( \sum_{j=0}^{\ell }(-c_{2-j}^{-}+id_{2-j}^{-})r^{3-j}e^{-i(1-j)t} \right)dt .
\end{eqnarray*}

We define the following functions:
\begin{eqnarray*}
    A_{2-j}^{-}(t )&=&-c_{2-j}^{-}\cos ((1-j)t)+d_{2-j}^{-}\sin ((1-j)t ),\\
    B_{2-j}^{-}(t)&=&c_{2-j}^{-}\sin ((1-j)t)+d_{2-j}^{-}\cos ((1-j)t),\\
    C_{2-j}^{-}(t )&=& -c_{2-j}^{-}\sin ((j-1)t)+d_{2-j}^{-}\cos ((j-1)t).
\end{eqnarray*}

With this notation, we have
\begin{align*}
(-c_{2-j}^{-} + i d_{2-j}) e^{-i(1-j)t} 
&=A_{2-j}^{-}(t)+iB_{2-j}^{-}(t),
\end{align*}

as a result,
\begin{eqnarray*}
    \Re \left( \int_{L_{h}^{-}}i\overline{\tilde{h}^{-}}dw \right)&=&\int_{\pi+\theta_{r}}^{2\pi -\theta_{r}}\left(\sum_{j=0}^{\ell } r^{3-j}A_{2-j}^{-}(t) \right)dt\\
   &=& \int_{\pi+\theta_{r}}^{2\pi -\theta_{r}}\left(r^{3}A_{2}^{-}(t)+r^{2}A_{1}^{-}(t)+\sum_{j=2}^{\ell } r^{3-j}A_{2-j}^{-} (t)\right)dt\\
   &=&\left[-c_{2}^{-}r^3\sin t - d_{2}^{-}r^{3}\cos t-c_{1}^{-}r^{2}t+\sum_{j=2}^{\ell }\dfrac{C_{2-j}^{-}(t)}{j-1}r^{3-j}\right]_{\pi +\theta_{r}}^{2\pi -\theta_{r}}
\end{eqnarray*}

Using the identities $\sin (\ell (2\pi -x))=-\sin(\ell x)$, $\cos (\ell (2\pi-x))=\cos (\ell x)$, $\sin (\ell (\pi+x))=(-1)^{\ell }\sin (\ell x)$ and $\cos (\ell (\pi+x))=(-1)^{\ell }\cos (\ell x)$, we obtain
\begin{eqnarray*}
\Re \left( \int_{L_{h}^{-}}i\overline{\tilde{h}^{-}}dw\right)&=&-2d_{2}^{-}r^{3}\cos (\theta_{r})-c_{1}^{-}r^{2}(2\theta_{r}-\pi )+\sum_{j=1}^{\lfloor \frac{ \ell -1}{2}\rfloor}\dfrac{c_{1-2j}}{j}r^{2-2j}\sin(2j\theta_{r})\\ &+&\sum_{j=1}^{\lfloor \frac{\ell }{2}\rfloor}\dfrac{2d_{2-2j}^{-}}{2j-1}r^{3-2j}\cos ((2j-1)\theta_{r}).
\end{eqnarray*}

Now making use of the identities $\sin (-\ell x)=-\sin (\ell x)$ and $\cos (-\ell x)=\cos (\ell x)$, follows that 
\begin{eqnarray*}
    \Re \left( \int_{L_{h}^{c_{1}}}i\overline{\tilde{h}^{C}}dw\right)&=&-2c_{2}^{C}r^{3}\sin (\theta_{r} )+2c_{1}^{C}r^{2} \theta_{r}-2\sum_{j=2}^{\ell }\dfrac{c_{2-j}^{C}}{j-1}r^{3-j}\sin((j-1)\theta_{r}).
\end{eqnarray*}

In order to compute the next integral, we now use the identities $\sin (\ell (\pi -x))=(-1)^{\ell +1}\sin (\ell x)$ and $\cos (\ell (\pi -x))=(-1)^{\ell }\cos (\ell x)$. Applying these, we obtain
\begin{eqnarray*}
     \Re \left( \int_{L_{h}^{+}}i\overline{\tilde{h}^{+}}dw\right)&=&2d_{2}^{+}r^{3}\cos(\theta_{r})-c_{1}^{+}r^{2}(2\theta_{r}-\pi)+\sum_{j=1}^{\lfloor \frac{ \ell -1}{2}\rfloor}\dfrac{c_{1-2j}^{+}}{j}r^{2-2j}\sin(2j\theta_{r})\\
     &-&\sum_{j=1}^{\lfloor \frac{\ell }{2}\rfloor}\dfrac{2d_{2-2j}^{+}}{2j-1}r^{3-2j}\cos ((2j-1)\theta_{r}).
\end{eqnarray*}

Finally, to compute the last integral, we make use of the identities $\sin \left(\ell (\pi -x)\right)=(-1)^{\ell+1}\sin \left (\ell \pi \right)$, $\sin\left( \ell (\pi +x)\right)=(-1)^{\ell }\sin \left( \ell x\right) $ and $\cos \left( \ell (\pi -x\right) =(-1)^{\ell }\cos \left( \ell x\right) =\cos \left( \ell (\pi +x)\right) $. With these, we find
\begin{eqnarray*}
    \Re \left( \int_{L_{h}^{c_{2}}}i\overline{\tilde{h}^{C}}dw\right)&=&2c_{2}^{C}r^{3}\sin (\theta_{r})+2c_{1}^{C}r^{2}\theta_{r}+2\sum_{j=2}^{\ell }\dfrac{c_{2-j}^{C}}{j-1}(-1)^{j}r^{3-j}\sin((j-1)\theta_{r}).
\end{eqnarray*}

Therefore, the Melnikov function is given by
\begin{eqnarray*}
    M(r)&=&D_{2}r^{3}\cos (\theta_{r})+C_{1}r^{2}(2\theta_{r}-\pi )-2D_{1}r^{2}\theta_{r}+\sum_{j=1}^{\lfloor \frac{ \ell -1}{2}\rfloor}C_{1-2j}r^{2-2j}\sin(2j\theta_{r})\\ &+&\sum_{j=1}^{\lfloor \frac{\ell }{2}\rfloor} D_{2-2j}r^{3-2j}\cos((2j-1)\theta_{r})
\end{eqnarray*}
where
\begin{align*}
    D_{2} &= -2d_{2}^{-} + 2d_{2}^{+}, & 
    D_{1} &= -c_{1}^{-} - c_{1}^{+}, & 
    C_{1} &= -4c_{1}, \\
    C_{1-2j} &= \frac{c_{1-2j}^{-} + c_{1-2j}^{+} - 2c_{1-2j}^{c}}{j}, & 
    D_{1-2j} &= \frac{2d_{2-2j}^{-} - 2d_{2-2j}^{+}}{2j - 1}.
\end{align*}

In particular, when $\ell =4$, we have
\begin{eqnarray*}
    M(r)=D_{2}f_{1}+D_{1}f_{2}+D_{1}f_{3}+C_{-1}f_{4}+D_{0}f_{5}+D_{-2}f_{6}
\end{eqnarray*}
where
\begin{align*}
    f_{1} &= r^{3}\cos(\theta_{r}), & f_{2} &= r^{2}(2\theta_{r}-\pi), & f_{3} &= r^{2}\theta_{r}, \\
    f_{4} &= \sin(2\theta_{r}), & f_{5} &= r\cos(\theta_{r}), & f_{6} &= r^{-1}\cos(3\theta_{r}).
\end{align*}
We have $W(f_{1},\dots,f_{6})(2)=-\frac{\sqrt{3}\pi}{64}\neq0$. Hence, the functions $f_{1},\dots ,f_{6}$ are linearly independent. Moreover, $f_{3}$ has a constant sign in the interval $\left( \frac{1}{2},\infty \right)$, and the coefficients of $f_{1},\dots ,f_{6}$ can be freely chosen. As a result, Lemma \ref{lem1} guarantees that it is possible to choose the coefficients in such a way that $M(r)$ has at least $5$ simple zero. Consequently, system (\ref{eq.sec3.1}) has at least $5$ limit cycles bifurcating from the center $\dot{w}=iw$.

Thus, system (\ref{eq.it2}) has at least $5$ limit cycles bifurcating from the center $\dot{z}=-i(z-1)$.
\end{proof}
\renewcommand{\proofname}{Proof of Theorem \ref{teoc}}
\begin{proof}
With the change of variable $w=\frac{-2zi}{z-1}$, we rewrite system (\ref{eq.it3}) as
\begin{equation}\label{eq.sec3.2}
\dot{w}=-i(w+2i)+\varepsilon \tilde{h}^{\sigma}(w), \quad w\in\Sigma_{\sigma},\quad\sigma=+,c,-
\end{equation}
where 
\begin{equation*}
        \tilde{h}^{\sigma }(w) = \sum_{j=0}^{3} \left( \frac{b_{j}^{\sigma} - i a_{j}^{\sigma}}{2} \right) w^{j} (w+2i)^{2 - j}, 
 \end{equation*}
Note that, for $\varepsilon =0$,  system (\ref{eq.sec3.2}) is Hamiltonian with
\begin{eqnarray*}
    H^{+}(x,y)=H^{c}(x,y)=H^{-}(x,y)=\dfrac{x^2+(y+2)^2}{2}.
\end{eqnarray*}

Moreover, system (\ref{eq.sec3.2})  has two families of periodic orbits. The first one cross the line $\Im(w)=-1$, denoted by $l_{h}=l_{h}^{-}\cup l_{h}^{c}$ for $h\in (\frac{1}{2},\frac{9}{2})$, were
\begin{eqnarray*}
& &l^{-}_{h}=\left\{ \sqrt{2h}e^{it}-2i, \pi -\arcsin \left( \frac{1}{\sqrt{2h}}\right) \leq t\leq 2\pi+\arcsin\left( \dfrac{1}{\sqrt{2h}}\right)\right\} ,\\
& &l^{c}_{h}=\left\{ \sqrt{2h}e^{it}-2i , \arcsin \left( \frac{1}{\sqrt{2h}}\right) \leq t\leq \pi-\arcsin\left( \dfrac{1}{\sqrt{2h}}\right)\right\} .
\end{eqnarray*}

The second family crosses the lines $\Im(w)=1$ and $\Im(w)=-1$, and is given by $L_{h}=L^{-}_{h}\cup L_{h}^{c_{1}}\cup  L_{h}^{+}\cup L_{h}^{c_{2}}$ for $h\in (\frac{9}{2},\infty )$, where 
\begin{eqnarray*}
& &L^{-}_{h}=\left\{ \sqrt{2h}e^{it}-2i, \pi -\arcsin \left( \frac{1}{\sqrt{2h}}\right) \leq t\leq 2\pi+\arcsin\left( \dfrac{1}{\sqrt{2h}}\right)\right\} ,\\
& &L^{c_{1}}_{h}=\left\{ \sqrt{2h}e^{it}-2i , \arcsin \left( \frac{1}{\sqrt{2h}}\right) \leq t\leq \arcsin\left( \dfrac{3}{\sqrt{2h}}\right)\right\},\\
& &L^{-}_{h}=\left\{ \sqrt{2h}e^{it}-2i, \arcsin \left( \frac{3}{\sqrt{2h}}\right) \leq t\leq \pi-\arcsin\left( \dfrac{3}{\sqrt{2h}}\right)\right\} ,\\
& &L^{c_{1}}_{h}=\left\{ \sqrt{2h}e^{it}-2i ,\pi- \arcsin \left( \frac{3}{\sqrt{2h}}\right) \leq t\leq \pi-\arcsin\left( \dfrac{1}{\sqrt{2h}}\right)\right\}.
\end{eqnarray*}

We consider two distinct families of periodic orbits of systems (\ref{eq.sec3.2}). The first family intersects the line $\Im(w)=-1$ at $a^{-}(h)=(-\sqrt{2h-1},-1)$ and $b^{-}(h)=(\sqrt{2h-1},-1)$. The second family intersects the line $\Im(w)=-1$ at $A^{-}(h)=(-\sqrt{2h-1},-1)$, $B^{-}(h)=(\sqrt{2h-1},-1)$ and the line $\Im(w)=1$ at $A^{+}(h)=(-\sqrt{2h-9},1)$, $B^{+}(h)=(\sqrt{2h-9},1)$.

The Assumptions $1$ and $2$ of Theorem \ref{melnikov} are satisfied for the family of periodic orbits that crosses both lines $\Im(w)=1$ and $\Im(w)=-1$, while for the family of orbits that crosses only the line $\Im(w)=-1$, the assumptions of Remark \ref{obs.1} are satisfied.

To estimate the number of limit cycles bifurcating from the center $\dot{w}=-i(w+2i)$, we need to analyze the Melnikov functions associated with each of families. Since $H^{+}(x,y)=H^{C}(x,y)=H^{-}(x,y)$, we have the following expressions:

For the family that crosses both lines, the expression for the Melnikov functions is given by
\begin{eqnarray*}
M(h)=\Re \left(\int_{L_{h}^{-}}i\overline{\tilde{h}^{-}}dz\right)+\Re \left(\int_{L_{h}^{c_{1}}}i\overline{\tilde{h}^{c}}dz\right)+\Re \left(\int_{L_{h}^{+}}i\overline{\tilde{h}^{+}}dz\right)+\Re \left(\int_{L_{h}^{c_{2}}}i\overline{\tilde{h}^{c}}dz\right).
\end{eqnarray*}

For the family that crosses only the line $\Im(w)=-1$, the Melnikov function is given by
\begin{eqnarray*}
    M_{1}(h)=\Re \left(\int_{l_{h}^{-}}i\overline{\tilde{h}^{-}}dz\right)+\Re \left(\int_{l_{h}^{c}}i\overline{\tilde{h}^{c}}dz\right).
\end{eqnarray*}

We will work with each of the functions separately. We will star with $M_{1}$, we will also need computed each of the integrals separately, observe that
\begin{eqnarray*}
    i\overline{\tilde{h}^{-}}(w)=\frac{\overline{w}^{3}(w+2i)}{|w+2i|^{2}} (A_{3}^{-}+iB_{3}^{-})+\sum_{j=0}^{2}(A_{j}^{-}+iB_{j}^{-})\overline{w}^{j},
\end{eqnarray*}
where
\begin{eqnarray*}
  & &A_{3}=\frac{b_{3}^{-}}{2},\: B_{3}=\frac{a_{3}^{-}}{2},\:   A_{2}^{-}=\frac{-(a_{0}^{-}+a_{1}^{-}+a_{2}^{-})}{2},\: B_{2}^{-}=\frac{b_{0}^{-}+b_{1}^{-}+b_{2}^{-}}{2},\\
  & &A_{1}=\frac{4b_{0}^{-}+2b_{1}^{-}}{2},\: B_{1}^{-}=\frac{4a_{0}^{-}+2a_{1}}{2}, \: A_{0}^{-}=4a_{0},\: B_{0}^{-}=-4b_{0}^{-}.
\end{eqnarray*}

To simplify the writing, we set $r=\sqrt{2h}$, $\theta^{1}_{r}=\arctan(\frac{1}{r})$, and  $\theta^{2}_{r}=\arctan(\frac{3}{r})$, we have that
\begin{eqnarray*}
    \int_{l_{h}^{-}}i\overline{\tilde{h}^{-}}(w)dw&=&\int_{l_{h}^{-}}\left( \frac{\overline{w}^{3}(w+2i)}{|w+2i|^{2}} (A_{3}^{-}+iB_{3}^{-})+\sum_{j=0}^{2}(A_{j}^{-}+iB_{j}^{-})\overline{w}^{j}\right) dw\\
    &=&\int_{\pi -\theta_{r}}^{2\pi +\theta_{r}}\left[ (A_{3}^{-}+iB_{3}^{-})\frac{(re^{-it}+2i)^{3}e^{it}}{r}+\sum_{j=0}^{2}(A_{j}^{-}+iB_{j}^{-})(re^{-it}+2i)^{j}\right]ire^{it}dt\\
    &=&\int_{\pi -\theta_{r}}^{2\pi +\theta_{r}} \left[ r^{3}e^{-it}(C_{3}^{-}+iD_{3}^{-}) + r^{2}(C_{2}^{-}+iD_{2}^{-}) \right. \\
   &&\left.+ re^{it}(C_{1}^{-}+iD_{1}^{-}) + e^{2it}(C_{0}+iD_{0}) \right] dt  
\end{eqnarray*}

where
\begin{eqnarray*}
  C_{3}^{-} &=& -B_{3}^{-}-B_{2}^{-}, \quad D_{3}^{-} = A_{3}^{-}+A_{2}^{-}, \quad C_{2}^{-} = -6A_{3}^{-}-4A_{2}^{-}-B_{1}^{-}, \\
  D_{2}^{-} &=& -6B_{3}^{-}-4B_{2}^{-}+A_{2}^{-}, \quad C_{1}^{-} = 12B_{3}^{-}+4B_{2}^{-}-2A_{1}^{-}-B_{0}^{-}, \\
  D_{1}^{-} &=& -12A_{3}^{-}-4A_{2}^{-}-2B_{1}^{-}+A_{0}^{-}, \quad C_{0}^{-} = 8A_{3}^{-}, \quad D_{0}^{-} = 8B_{3}^{-}.
\end{eqnarray*}

Thus,
\begin{eqnarray*}
    \Re \left( \int_{l_{h}^{-}}i\overline{\tilde{h}^{-}}(w)dw\right)
    &=& \int_{\pi -\theta_{r}}^{2\pi +\theta_{r}} \left[ C_{3}^{-}r^{3}\cos t + D_{3}^{-}r^{3}\sin t  +  C_{2}^{-}r^{2}\right. \\
    && \left.  + C_{1}^{-}r\cos t - D_{1}^{-}r\sin t + C_{0}^{-}\cos 2t - D_{0}^{-}\sin 2t \right] dt\\
    &=& \left[ C_{3}^{-}r^{3}\sin t - D_{3}^{-}r^{3}\cos t + C_{2}^{-}r^{2}t + C_{1}^{-}r\sin t \right. \\
    && \left. + D_{1}^{-}r\cos t + \frac{C_{0}^{-}}{2}\sin 2t + \frac{D_{0}^{-}}{2}\cos 2t \right]_{\pi -\theta_{r}}^{2\pi +\theta_{r}}
\end{eqnarray*}

Using that $\cos\left(2\pi+x\right)=\cos x$, $\cos\left(\pi-x\right)=-\cos x$ and $\sin\left(2\pi+x\right)=\sin\left(\pi-x\right)=\sin x$, we obtain that 
\begin{eqnarray*}
 \Re \left( \int_{l_{h}^{-}}i\overline{\tilde{h}^{-}}(w)dw\right)=-2D_{3}^{-}r^{3}\cos \theta_{r}^{1}+C_{2}^{-}r^{2}(\pi +2\theta_{r}^{1})+2D_{1}^{-}r\cos \theta_{r}^{1}+C_{0}^{-}\sin 2\theta_{r}^{1}
\end{eqnarray*}
and
\begin{eqnarray*}
    \Re\left( \int_{l_{h}^{c}}i\overline{\tilde{h}^{c}}(w)dw\right)&=&2D_{3}^{c}r^{3}\cos \theta_{r}^{1}+C_{2}^{c}r^{2}(\pi-2\theta_{r}^{1})-2D_{1}^{c}r\cos \theta_{r}^{1}-C_{0}^{c}\sin 2\theta_{r}^{1}.
\end{eqnarray*}

Therefore,
\begin{eqnarray*}
    M_{1}(r)&=&(2D_{3}^{c}-2D_{3}^{-})r^{3}\cos \theta_{r}^{1}+C_{2}^{-}r^{2}(\pi +2\theta_{r}^{1})+C_{2}^{c}r^{2}(\pi -2\theta_{r}^{1})+(2D_{1}^{-}-2D_{1}^{c})r\cos \theta_{r}^{1}\\
    &+&(c_{0}^{-}-C_{0}^{c})\sin 2\theta_{r}^{1}.
\end{eqnarray*}

Note that the functions $f_{1}=r^{3}\cos \theta_{r}^{1}$, $f_{2}=r^{2}(\pi +2\theta_{r}^{1})$, $f_{3}=r^{2}(\pi -2\theta_{r}^{1})$, $f_{4}=r\cos \theta_{r}^{1}$ and $f_{5}=\sin 2\theta_{r}^{1}$ are linearly independent functions, since $W(f_{1},f_{2},f_{3},f_{4},f_{5})(1)=-384\pi \neq 0$. Moreover $f_{2}$ has constant sign on the $\left(\frac{1}{2},\frac{9}{2} \right)$. Since the coefficients of the functions $f_{1},...f_{5}$ can be  chosen arbitrarily, by Lemma \ref{lem1}, there exist constants such that $M_{1}(r)$  has at least $4$ simple zeros.

Having completed the analysis of $M_{1}$, we now turn to $M(r)$. We have that
\begin{eqnarray*}
    \Re\left( \int_{L_{h}^{-}}i\overline{\tilde{h}^{-}}(w)dw\right)&=&-2D_{3}^{-}r^{3}\cos \theta_{r}^{1}+C_{2}^{-}r^{2}(\pi +2\theta_{r}^{1})+2D_{1}^{-}r\cos \theta_{r}^{1}+C_{0}^{-}\sin 2\theta_{r}^{1},\\
    \Re\left( \int_{L_{h}^{c_{1}}}i\overline{\tilde{h}^{c}}(w)dw\right)&=& D_{3}^{c}r^{3}(\cos \theta_{r}^{1}-\cos \theta_{r}^{2})+C_{3}^{c}r^{3}(-\sin \theta_{r}^{1}+\sin \theta_{r}^{2})+C_{2}^{c}r^{2}(-\theta_{r}^{1}+\theta_{r}^{2})\\
    &-&D_{1}^{c}r(\cos \theta_{r}^{1}-\cos \theta_{r}^{2})+C_{1}^{c}r(-\sin \theta_{r}^{1}+\sin \theta_{r}^{2})-\frac{D_{0}^{c}}{2}(\cos 2\theta_{r}^{1}-\cos 2\theta_{r}^{2})\\
    &+&\frac{C_{0}^{c}}{2}(-\sin 2\theta_{r}^{1}+\sin 2\theta_{r}^{2}),\\
    \Re\left( \int_{L_{h}^{+}}i\overline{\tilde{h}^{+}}(w)dw\right)&=&2D_{3}^{+}r^{3}\cos \theta_{r}^{2}+C_{2}^{+}r^{2}(\pi-2\theta_{r}^{2})-2D_{1}^{+}r\cos \theta_{r}^{2}-C_{0}^{+}\sin 2\theta_{r}^{2},\\
    \Re\left( \int_{L_{h}^{c_{2}}}i\overline{\tilde{h}^{c}}(w)dw\right)&=&D_{3}^{c}r^{3}(\cos \theta_{r}^{1}-\cos \theta_{r}^{2})+C_{3}^{c}r^{3}(\sin \theta_{r}^{1}-\sin \theta_{r}^{2})+C_{2}^{c}r^{2}(-\theta_{r}^{1}+\theta_{r}^{2})\\
    &-&D_{1}^{c}r(\cos \theta_{r}^{1}-\cos \theta_{r}^{2})+C_{1}^{c}r(\sin \theta_{r}^{1}-\sin \theta_{r}^{2})+\frac{D_{0}^{c}}{2}(\cos 2\theta_{r}^{1}-\cos 2\theta_{r}^{2})\\
    &+&\frac{C_{0}^{c}}{2}(-\sin 2\theta_{r}^{1}+\sin 2\theta_{r}^{2})
\end{eqnarray*}

Therefore
\begin{eqnarray*}
    M(r)=\alpha_{1}g_{1}+\alpha_{2}g_{2}+\alpha_{3}g_{3}+\alpha_{4}g_{4}+\alpha_{5}g_{5}+\alpha_{6}g_{6}+\alpha_{7}g_{7}+\alpha_{8}g_{8}+\alpha_{9}g_{9},
\end{eqnarray*}
with
\begin{alignat*}{3}
g_{1} &= r^{3}\cos \theta_{r}^{1}, \quad & g_{2} &= r^{2}(\pi + 2\theta_{r}^{1}), \quad & g_{3} &= r\cos \theta_{r}^{1}, \\
g_{4} &= \sin 2\theta_{r}^{1}, \quad & g_{5} &=  r^{3}\cos \theta_{r}^{2}\quad & g_{6} &=  r^{2}(\pi - 2\theta_{r}^{2})\\
g_{7} &=  r\cos \theta_{r}^{2}, \quad & g_{8} &= \sin 2\theta_{r}^{2}, \quad & g_{9} &= r^{2}(-\theta_{h}^{1}+\theta_{r}^{2}),
\end{alignat*}
and
\begin{alignat*}{3}
\alpha _{1} &=-2D_{3}^{-}+2D_{3}^{c}, \quad & \alpha_{2} &=C_{2}^{-}, \quad & \alpha_{3} &= 2D_{1}^{-}-2D_{1}^{c}, \\
\alpha_{4} &= C_{0}^{-}-C_{0}^{c}, \quad & \alpha_{5} &= 2D_{3}^{+}-2D_{3}^{c} ,\quad & \alpha_{6} &=  C_{2}^{+},\\
\alpha_{7} &= - 2D_{1}^{+}+2D_{1}^{c}, \quad & \alpha_{8} &=- C_{0}^{+}+C_{0}^{c}, \quad & \alpha_{9} &= 2C_{2}^{c}.
\end{alignat*}

We have that $W(g_{1},\cdots ,g_{9})(5)\approx 0.55155$, which implies that the functions $g_{1},...,g_{9}$ are linearly independent. Furthermore, $g_{2}$ has a constant sign in $\left(\frac{9}{2},\infty \right)$. Therefore, $M$ has at least $8$ simple zeros.
\end{proof}

\section{Maximum number of limit cycles of certain piecewise holomorphic systems}
\renewcommand{\proofname}{Proof of statement (a) of Theorem \ref{teod}}
\begin{proof}
Based on Proposition \ref{lc2}, it follows that the solutions of the equations  $\dot{z}=\frac{1}{z_{1}^{+}(z-z_{0}^{+})}$, $\dot{z}=\frac{1}{z_{1}^{-}(z-z_{0}^{-})}$ and $\dot{z}=\frac{1}{\sum_{j=0}^{n}z_{j}^{c}z^{j}}$ with $z_{1}^{\pm}=a^{\pm }+ib^{\pm} $ and $z_{0}^{\pm }=x_{0}^{\pm }+iy_{0}^{\pm }$ are contained in the level curves of the functions
\begin{eqnarray*}
    H^{+}(x,y)&=&\dfrac{-b^{+}x^{2}-2a^{+}xy+b^{+}y^{2}}{2}+(-b^{+}y_{0}^{+}+a^{+}x^{+}_{0})y+(a^{+}y_{0}^{+}+b^{+}x_{0}^{+})x,\\
    H^{-}(x,y)&=&\dfrac{-b^{-}x^{2}-2a^{-}xy+b^{-}y^{2}}{2}+(-b^{-}y_{0}^{-}+a^{-}x^{-}_{0})y+(a^{-}y_{0}^{-}+b^{-}x_{0}^{-})x,\\
    H^{C}(x,y)&=&p^{c}(x,y),
\end{eqnarray*}
respectively, with $p(x,y)$  a polynomial of degree $n+1$. The polynomial $p$ is defined as the imaginary part of the anti-derivate of the $p^{c}(z)$. Its explicit form can be obtained by applying the Binomial Theorem to expand the powers of $z.$

 Suppose that the system $(C1)$ has a limit cycle $\gamma $ that intersects the line $\Im (z)=-1$ at the points $(s_{1},-1)$ and $(s_{2},-1)$, with $s_{1}\neq s_{2}$, and intersects the line $\Im (z)=1$ at the points $(t_{1},1)$ and $(t_{2},1)$, with $t_{1}\neq t_{2}$. We then have
 \begin{equation}\label{eq1.teod}
    \begin{cases}
        H^{-}(s_{1},-1)-H^{-}(s_{2},-1)=0,\\
        H^{C}(s_{2},-1)-H^{C}(t_{2},1)=0,\\
        H^{+}(t_{2},1)-H^{+}(t_{1},1)=0,\\
        H^{C}(t_{1},1)-H^{C}(s_{1},-1)=0.\\
    \end{cases}
\end{equation}

From the first equation of (\ref{eq1.teod}), we have that
\begin{eqnarray*}
    \frac{b^{-}}{2}(s_{2}^{2}-s_{1}^{2})-a^{-}(s_{2}-s_{1})-(a^{-}y_{0}^{-}+b^{-}x_{0}^{-})(s_{2}-s_{1})=0.
\end{eqnarray*}

Since $s_{1}\neq s_{2}$, we can divide the expression adobe by $s_{2}-s_{1}$, obtaining
\begin{eqnarray*}
    \frac{b^{-}}{2}(s_{2}+s_{1})-a^{-}-(a^{-}y_{0}^{-}+b^{-}x_{0}^{-})=0,
\end{eqnarray*}
it then follows that
\begin{eqnarray}\label{eq2.teod}
    s_{2}=\frac{2(((a^{-}y_{0}^{-}+b^{-}x_{0}^{-})+a^{-})}{b^{-}}-s_{1}.
\end{eqnarray}

Similarly, from the third equation of (\ref{eq1.teod}), we  obtain
\begin{eqnarray}\label{eq3.teod}
    t_{2}=\frac{2(((a^{+}y_{0}^{+}+b^{+}x_{0}^{+})+a^{+})}{b^{+}}-t_{1}.
\end{eqnarray}

It is important to note that that the second and fourth equations of (\ref{eq1.teod}) represent the same polynomial of $n+1$, but expressed in different sets of variables, the second equation is in $s_{2}$ and $t_{2}$, while the fourth equation is in $s_{1}$ and $t_{1}$. 

By substituting the expressions for $s_{2}$ and $t_{2}$, obtained in equations (\ref{eq2.teod}) and (\ref{eq3.teod}), respectively, into the second equation of (\ref{eq1.teod}), we obtain a new equation solely in terms of $s_{1}$ and $t_{1}$. Consequently, the original system can be reduced to one involving only these two variables.

By subtracting the second equation from the fourth equation of the system (\ref{eq1.teod}), we obtain an equivalent system in which the new second equation is a polynomial of degree $n$, while the forth equation remains unchanged, with degree $n+1$. Applying Bézout's Theorem (see \cite{refb25}) to these two equations, we conclude that, when the system has finitely many solutions, the maximum number of solutions is $n(n+1)$.

Note that if $(x_{0},y_{0})$ is a solution  of the system formed by the second and fourth equation of (\ref{eq1.teod}), then the point 
\begin{eqnarray*}
    (x_{1},y_{1})=\left(\frac{2(((a^{-}y_{0}^{-}+b^{-}x_{0}^{-})+a^{-})}{b^{-}}-x_{0},\frac{2(((a^{+}y_{0}^{+}+b^{+}x_{0}^{+})+a^{+})}{b^{+}}-y_{0}\right),
\end{eqnarray*}
is also a solution of the same system.

Therefore, the maximum number of limit cycles of system $(C1)$ is $\frac{n(n+1)}{2}$.
\end{proof}

\renewcommand{\proofname}{Proof of statement (b) of Theorem \ref{teod}}
\begin{proof}
    From item a), for $n=1$, class $(C1)$ has at most one limit cycle. Moreover  the system 
\begin{eqnarray}\label{exemplo.a.teod}
\begin{cases}
\dot{z}=\frac{1}{-i(z-(\frac{9}{4}+6i))}, \quad \hbox{if} \quad z\in \Sigma_{+},\\
\dot{z}=\frac{1}{(1+i)(z-(\frac{1}{2}+i))},\quad \hbox{if} \quad z\in \Sigma_{c},  \\
\dot{z} =\frac{1}{i(z+4i)},\quad \hbox{if} \quad z\in \Sigma_{-},
\end{cases}
\end{eqnarray}
has a limit cycle.

Indeed, initially, note that the orbits of the equations $\dot{z}=\frac{1}{z-(\frac{9}{4}+6i)}$, $\dot{z}=\frac{1}{(1+i)(z-(\frac{1}{2}-i))}$ and $\dot{z} =\frac{1}{i(z+4i)}$ are contained within the level curves of the functions
\begin{eqnarray*}
    H^{+}(x,y)&=&\frac{x^{2}-y^{2}}{2}+6y-\frac{9}{4}x,\\
    H^{C}(x,y)&=&\frac{-x^{2}+y^{2}+2xy+y+3x}{2},\\
    H^{-}(x,y)&=&\frac{-x^{2}+y^{2}}{2}+4y.
\end{eqnarray*}

In this case, system (\ref{eq1.teod}) reduces to
\begin{eqnarray*}
\begin{cases}
    s_{2}^{2}-s_{1}^{2}=0,\\
    s_{2}-s_{2}^{2}-5t_{2}+t_{2}^{2}-2=0,\\
    \frac{1}{2}(t_{2}^{2}-t_{1}^{2})-\frac{9}{4}(t_{2}-t_{1})=0,\\
    s_{1}-s_{1}^{2}-5t_{1}+t_{1}^{2}-2=0.\\
    \end{cases}
\end{eqnarray*}

Hence,
\begin{eqnarray*}
    s_{2}=-s_{1}\quad \hbox{and} \quad t_{2}=\frac{9}{2}-t_{1},
\end{eqnarray*}
and, according to the proof in item (a), we obtain the following system
\begin{eqnarray*}
    \begin{cases}
        -2s_{1}+t_{1}-\frac{9}{4}=0,\\
        s_{1}-s_{1}^{2}-5t_{1}+t_{1}^{2}-2=0,
    \end{cases}
\end{eqnarray*}
which admits the solutions
\begin{eqnarray*}
 (p_{1},q_{1})= (-1.652018966,-1.054037933)\; \hbox{and}\; (p_{2},q_{2})=(1.652018966,5.554037933).
\end{eqnarray*}

Therefore, the system has a limit cycle that passes through $(p_{1},-1),(p_{2},-1), (q_{1},1)$ and $(q_{2},1)$.
\begin{figure}[!htb]
\centering
\includegraphics[scale=0.6]{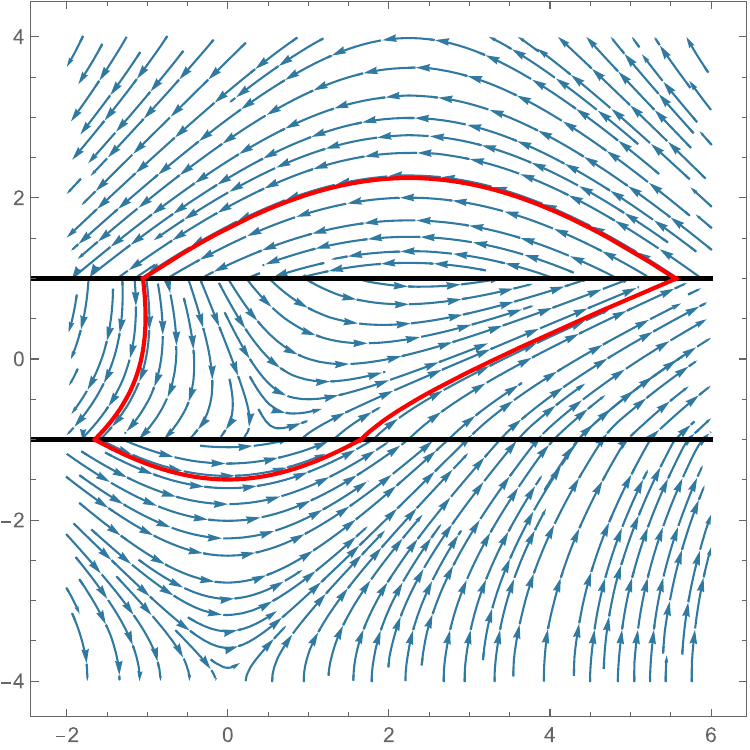}
\caption{Phase portrait of system (\ref{exemplo.a.teod}).}
\label{img20}
\end{figure}
\end{proof}

\renewcommand{\proofname}{Proof of statement (c) of Theorem \ref{teod}}
\begin{proof}
From item a), for $n=2$, class $(C1)$ has at most three limit cycle. Moreover  the system 
\begin{eqnarray}\label{exemplo.b.teod}
\begin{cases}
\dot{z}=\frac{1}{i(z-(\frac{47}{100}+6i))}, \quad \hbox{if}\quad z\in \Sigma_{+},\\
\dot{z}=\frac{1}{(1-i)(z-2i)(z-(3-2i))},\quad \hbox{if} \quad z\in \Sigma_{c},  \\
\dot{z} =\frac{1}{-i(z-(\frac{18}{10}-4i))},\quad \hbox{if} \quad z\in \Sigma_{-},
\end{cases}
\end{eqnarray}
has two limit cycles

Indeed, the orbits of the equations $\dot{z}=\frac{1}{i(z-(\frac{47}{100}+6i))}$, $\dot{z}=\frac{1}{(1-i)(z-2i)(z-(3-2i))}$ and $\dot{z}=\frac{1}{-i(z-(\frac{18}{10}-4i))}$
are contained in the level curves of the functions
\begin{eqnarray*}
    H^{+}(x,y)&=&\dfrac{-600y-47x+50x^{2}-50y^{2}}{100}\\
    H^{C}(x,y)&=&\dfrac{(2x^{3}-3x^{2}(7+2y)-6x(-2+y+y^{2})+y(60+21y+2y^{2}))}{6}\\
    H^{-}(x,y)&=&\dfrac{-40y+18x-5x^{2}+5y^{2}}{10}\\
\end{eqnarray*}

Accordingly, system (\ref{eq1.teod}) can be rewritten as
\begin{eqnarray*}
    \begin{cases}
        5(s_{2}^{2}-s_{1}^{2})-18(s_{2}-s_{1})=0,
        \\
        (-41 + 12 s_{2} - 15 s_{2}^{2} + 2 s_{2}^{3}) + (-83 + 27 t_{2}^{2} - 2 t_{2}^{3}) = 0,\\
        50(t_{2}^{2}-t_{1}^{2})-47(t_{2}-t_{1})=0,\\
        (-41 + 12 s_{1} - 15 s_{1}^{2} + 2 s_{1}^{3}) + (-83 + 27 t_{1}^{2} - 2 t_{1}^{3}) = 0.
    \end{cases}
\end{eqnarray*}

Thus,
\begin{eqnarray*}
    s_{2}=\frac{18}{5}-s_{1}\quad \hbox{and} \quad t_{2}=\frac{47}{50}-t_{1},
\end{eqnarray*}
using the same argument as in item (a), we arrive at the system
\begin{eqnarray*}
    \begin{cases}
        \frac{756 s_1}{25} - \frac{42 s_1^2}{5} + \frac{-8865374 - 1420575 t_1 + 1511250 t_1^2}{31250}=0,\\
        (-41 + 12 s_{1} - 15 s_{1}^{2} + 2 s_{1}^{3}) + (-83 + 27 t_{1}^{2} - 2 t_{1}^{3}) = 0, 
    \end{cases}
\end{eqnarray*}
whose solutions are given by
\begin{eqnarray*}
(-1.9801740022092,3.2995684935354)\; \hbox{and}\; (5.5801740022095,-2.3595684935353)\\
(-0.3333275510587,-2.0429434088287)\; \hbox{and}\; (3.9333275494283,2.9829434085994)\\
(-0.10029992808998,-2.0102455495196)\; \hbox{and}\; (3.7002992808998,29502455495196)
\end{eqnarray*}

The first pair of solutions cannot define a limit cycle. In order for a limit cycle to exit passing through the point $(p_{1},-1)$, $(p_{2},-1)$, $(q_{2},1)$ and $(q_{1},1)$, we must have the inequalities $p_{1}<p_{2}$ and $p_{1}<p_{2}$. However, in the first solution,
\begin{eqnarray*}
    (p_{1},q_{1})=(-1.980174,3.299568)\quad \hbox{and}\quad (p_{2},q_{2})=(5.580174,-2.359568),
\end{eqnarray*}
we observe that while $p_{1}<p_{2}$, the inequality $q_{1}<q_{2}$ does not hold. Therefore, the system has two limit cycles.

 \begin{figure}[!htb]
\centering
\includegraphics[scale=0.6]{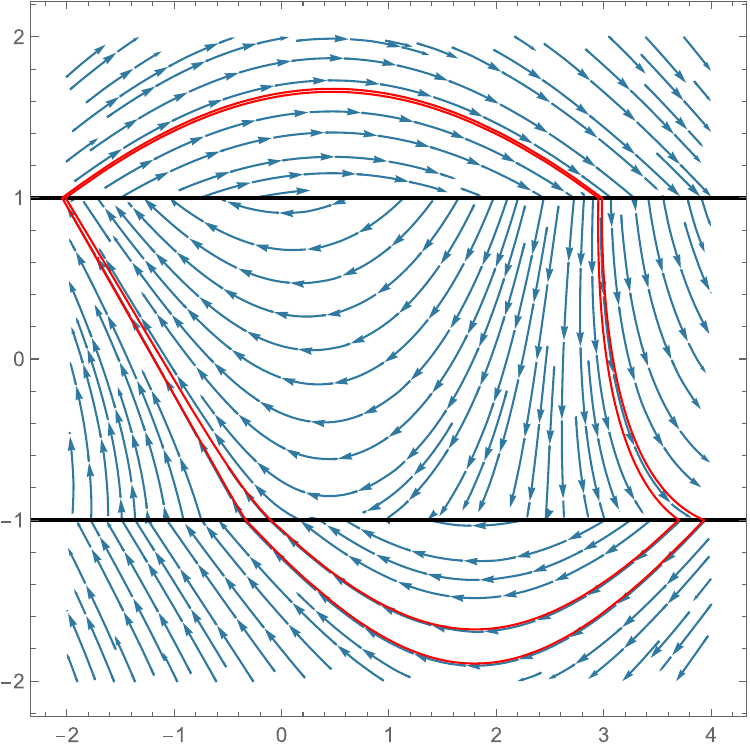}
\caption{Phase portrait of system (\ref{exemplo.b.teod}).}
\label{img21}
\end{figure}
\end{proof}

To conclude, we will prove Theorems \ref{teoe} . To establish item (a), we rely on an intermediate result, obtained based on the proof of Theorem 1 in \cite{refb5}.
\begin{lema}
  Consider a piecewise discontinuous planar system  formed by linear centers, where one of the discontinuity regions is $\mathbb{S}_{1}$. Then, the maximum number of crossing limit cycles that intersect $\mathbb{S}_{1}$ at two points is three.
\end{lema}
\begin{proof}
    In the region delimited by $\mathbb{S}_{1}$, consider the arbitrary linear center, as demonstrated in \cite{refb13},  which can be written as
    \begin{eqnarray*}
        \dot{x}=-bx-\dfrac{4b^{2}+\omega^{2}}{4a}y+d,\quad \dot{y}=ax+by+c,
    \end{eqnarray*}
    with $a>0$ and $\omega >0$. Then, its trajectories are contained in the level curves of the function
    \begin{eqnarray*}
        H(x,y)=4(ax+by)^{2}+8a(cx-dy)+\omega^{2}y^2. 
    \end{eqnarray*}
    
    Note that, after the scaling of time $\sigma=at$, we can assume that $a=1$.
    
    Suppose the discontinuous piecewise differential system where one of the regions of discontinuity is $\mathbb{S}_{1}$, delimiting a linear center has four crossing periodic orbits that intersect $\mathbb{S}_{1}$ in the points $(x_{i},y_{i})$ for $i=1,2$, so these points must satisfy 
    \begin{eqnarray*}
        H(x_{1},y_{1})-H(x_{2},y_{2})=0
    \end{eqnarray*}
    or equivalently
    \begin{eqnarray}\label{mnlc.lema}
        4(by_{1}+x_{1})^{2}-4(by_{2}+x_{2})^{2}+8c(x_{1}-x_{2})-8d(y_{1}-y_{2})+\omega^{2}(y_{1}^{2}-y_{2}^{2})=0.
    \end{eqnarray}

    Suppose the four solutions are $(p_{1},p_{2})$, $(q_{1},q_{2})$, $(t_{1},t_{2})$ and $(w_{1},w_{2})$ with $p_{i},q_{i},t_{i}$ and $w_{i}\in \mathbb{R}^{2}$, $p_{i}\neq p_{i}$, $q_{i}\neq q_{j} $, $t_{i}\neq t_{j}$ and $w_{i}\neq w_{j}$ for $i\neq j $ and $i,j=1,2$. Let us consider $p_{i}=(x_{i},y_{i})$, $q_{i}=(k_{i},l_{i})$, $t_{i}=(m_{i},n_{i})$, $w_{i}=(h_{i},j_{i})$, for $i=1,2$. 

    Substituting $(p_{1},p_{2})$ and $(q_{1},q_{2})$ into (\ref{mnlc.lema}), we obtain the coefficients $d$ and $c$. Substituting $(t_{1},t_{2})$ and $(w_{1},w_{2})$ into (\ref{mnlc.lema}), we obtain $\omega$.

    Since the points $p_{i}=(x_{i},y_{i})$, $q_{i}=(k_{i},l_{i})$, $t_{i}=(m_{i},n_{i})$, $w_{i}=(h_{i},j_{i})$ belong to $\mathbb{S}_{1}$, then $x_{i}^{2}=1-y_{i}^{2}$, $k_{i}^{2}=1-l_{i}^{2}$, $m_{i}^{2}=1-n_{i}^{2}$, $h_{i}^{2}=1-j_{i}^{2}$. Thus, we obtain $\omega=2$ and $b=0$, and it follows that $c=d=0$. Therefore, inside $\mathbb{S}_{1}$, we get the linear center
    \begin{eqnarray*}
        \dot{x}=-y, \quad \dot{y}=x,
    \end{eqnarray*}
  which leads to a contradiction. Thus, the maximum number of limit cycles the system can have is three.   
\end{proof}

    While this proof is based on the techniques developed in \cite{refb5}, similar ideas are also explored in \cite{refb14}.

\renewcommand{\proofname}{Proof of statement (b) of Theorem \ref{teoe}}
\begin{proof}
The system
\begin{eqnarray}\label{exemplo.b1.teoe}
\begin{cases}
\dot{z}=i(z-(1.0980423999-i0.80012406276)), \quad \hbox{if} \quad z\in \Sigma_{+}^{E},\\
\dot{z}=i(z-(1+i0.995016)),\quad \hbox{if} \quad z\in \Sigma_{c}^{E},  \\
\dot{z} =-i(z-(3.2900009902+i0.9400008902)),\quad \hbox{if} \quad z\in \Sigma_{-}^{E},
\end{cases}
\end{eqnarray}
has two limit cycles.

In fact, by letting $w=\frac{-2iz}{z-1}$, system (\ref{exemplo.b1.teoe}) is transformed into
\begin{eqnarray}\label{exemplo.b2.teoe}
    \begin{cases}
\dot{w}=-w(w(1-(1.0980423999-i0.80012406276))-2i),\quad \hbox{if}\quad w\in \Sigma_{+},\\
\dot{w}=-w(w(1-(1+i0.995016))-2i),\quad \hbox{if} \quad w\in \Sigma_{c},  \\
\dot{w} =w(w(1-(3.2900009902+i0.9400008902))-2i),\quad \hbox{if} \quad w\in \Sigma_{-}.
\end{cases}
\end{eqnarray}

Note that, in each region, it is given by
\begin{eqnarray*}
    \dot{w}=\pm w(w(1-z_{0}^{\sigma })-2i),
\end{eqnarray*}
where $z_{0}^{\sigma}=x_{0}^{\sigma }+iy_{0}^{\sigma}$,  depends on the region considered, $z_{0}^{+}=1.0980423999-i0.80012406276$, $z_{0}^{c}=1+i0.995016$ and $z_{0}^{-}=3.2900009902+i0.9400008902$. Also note that the solutions of previous equation are contained in the level curves of the function
\begin{eqnarray*}
    H^{\sigma }(x,y)=\frac{((1-x_{0}^{\sigma })x+y_{0}^{\sigma}y)^{2}+((1-x_{0}^{\sigma })y-y_{0}^{\sigma}x-2)^{2}}{x^{2}+y^{2}},
\end{eqnarray*}
for $\sigma =+,c,-$.

Now, we seek the limit cycles of the system (\ref{exemplo.b2.teoe}) that intersect the line $\Im(w)=-1$ at two distinct point $(s_{1},-1)$ and $(s_{2},-1)$, and the line $\Im(w)=1 $ at the points $(t_{1},1)$ and $(t_{2},1)$. These points must then satisfy the following conditions
\begin{eqnarray}\label{sistemateoe.b}
\begin{cases}
(s_{1} - s_{2}) \left( -s_{2} (-2 + x_{0}^{-}) - y_{0}^{-} + s_{1} \left( 2 - x_{0}^{-} + s_{2}  y_{0}^{-} \right) \right) = 0,\\
-4 \left( -2 + (2 + s_{1}^{2}) - y_{0}^{c} s_{1} + y_{0}^{c} (1 + s_{1}^{2}) t_{1} + t_{1}^{2} (-1 - y_{0}^{c} s_{1}) \right) = 0, \\
(t_{1} - t_{2}) \left( t_{2}  x_{0}^{+} + y_{0}^{+} + t_{1} \left( x_{0}^{+} - y_{0}^{+}  t_{2} \right) \right) = 0, \\
-4 \left( -2 + (2 + s_{1}^{2}) - y_{0}^{c} s_{1} + y_{0}^{c} (1 + s_{1}^{2}) t_{1} + t_{1}^{2} (-1 - y_{0}^{c} s_{1}) \right) = 0.
\end{cases}
\end{eqnarray}

Given that $s_{1}\neq s_{2}$ and $t_{1}\neq t_{2}$, it follows from the first and third equations.

\[
s_2 = \frac{(-2 + x_0^{-}) s_1 + y_0^{-}}{2 - x_0^{-} + y_0^{-} s_1}
\quad \text{e} \quad
t_2 = \frac{-x_0^{+} t_1 - y_0^{+}}{x_0^{+} - y_0^{+} t_1}
\]

Substitution the above expressions into the second equation in (\ref{sistemateoe.b}), we arrive at the following system
\begin{eqnarray*}
\begin{cases}
&-4 \left( -2 + (2 + s_{1}^{2}) - 0.995016\, s_{1} + 0.995016\, (1 + s_{1}^{2}) t_{1} + t_{1}^{2} (-1 - 0.995016\, s_{1}) \right) = 0, \\[1em]
&-4 \bigg( 
-2 + \left(2 + \left( \frac{1.2900009902 s_{1} + 0.9400008902}{-1.2900009902 + 0.9400008902 s_{1}} \right)^{2} \right) 
- 0.995016 \left( \frac{1.2900009902 s_{1} + 0.9400008902}{-1.2900009902 + 0.9400008902 s_{1}} \right) \\
&\quad + 0.995016 \left(1 + \left( \frac{1.2900009902 s_{1} + 0.9400008902}{-1.2900009902 + 0.9400008902 s_{1}} \right)^{2} \right)
\left( \frac{-1.0980423999 t_{1} - 0.80012406276}{1.0980423999 - 0.80012406276 t_{1}} \right) \\
&\quad + \left( \frac{-1.0980423999 t_{1} - 0.80012406276}{1.0980423999 - 0.80012406276 t_{1}} \right)^{2}
\left(-1 - 0.995016 \left( \frac{1.2900009902 s_{1} + 0.9400008902}{-1.2900009902 + 0.9400008902 s_{1}} \right) \right) 
\bigg) = 0,
\end{cases}
\end{eqnarray*}
for which the solutions are
\begin{eqnarray*}
   (p_{1},q_{1})= (-2.422768823,-2.422768823) \; \hbox{and}\; (p_{2},q_{2})=(0.6125945264,0.6125945264)\\
  (\tilde{p}_{1},\tilde{q}_{1})=  (-1.632401134,-1.632401134)\; \hbox{and}\; (\tilde{p}_{2},\tilde{q}_{2})=(0.4127508948,0.4127508948).
\end{eqnarray*}

Then, system (\ref{exemplo.b2.teoe}) has two limit cycles, one passing through $(p_{1},-1)$, $(p_{2},-1)$, $(q_{2},1)$ and $(q_{1},1)$, and another one passing through $(\tilde{p}_{1},-1)$, $(\tilde{p}_{2},-1)$, $(\tilde{q}_{2},1)$ and $(\tilde{q}_{1},1)$. Consequently, system (\ref{exemplo.b1.teoe}) also has two limit cycles.
 \begin{figure}[!htb]
\centering
\includegraphics[scale=0.6]{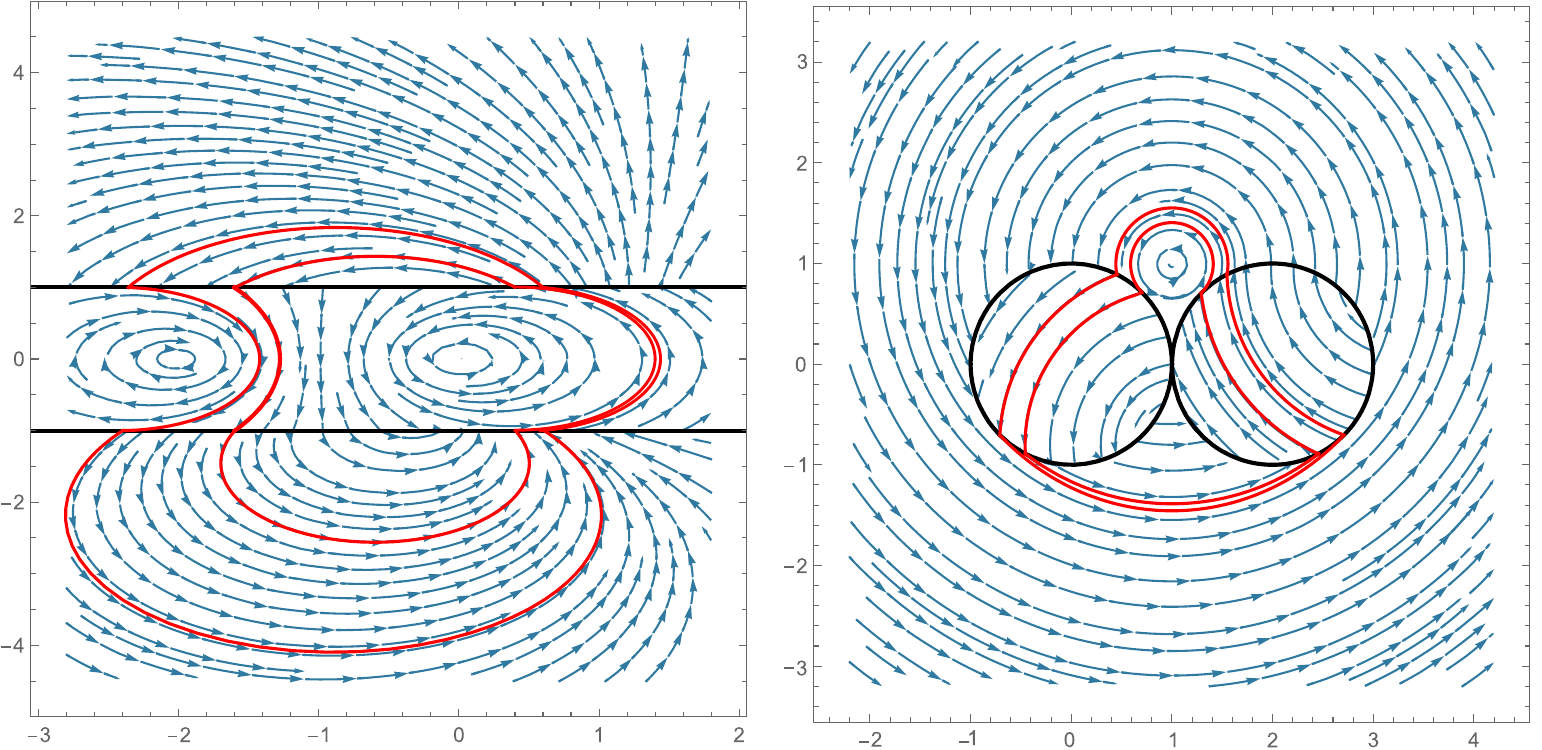}
\caption{Phase portrait of systems  (\ref{exemplo.b2.teoe}) and (\ref{exemplo.b1.teoe}).}
\label{imga10}
\end{figure}
\end{proof}

\renewcommand{\proofname}{Proof of statement (c) of Theorem \ref{teoe}}
\begin{proof}
The system
\begin{eqnarray}\label{exemplo.b1.teof}
\begin{cases}
\dot{z}=i(z-(-0.2-i0.6)), \quad \hbox{if} \quad z\in \Sigma_{+}^{I},\\
\dot{z}=-i(z-(1+i)),\quad \hbox{if} \quad z\in \Sigma_{c}^{I},  \\
\dot{z} =i(z-(-0.6-i0.4)),\quad \hbox{if}\quad z\in \Sigma_{-}^{I},
\end{cases}
\end{eqnarray}
has two limit cycles.

In fact, letting $w=-\frac{2ix}{z-1}$, system (\ref{exemplo.b1.teof}) is transformed into
\begin{eqnarray}\label{exemplo.b2.teof}
    \begin{cases}
\dot{w}=-(w+2i)(w(1-(-0.2-0.6i))-2i(-0.2-0.6i)),\quad \hbox{if} \quad w\in \Sigma_{+},\\
\dot{w}=(w+2i)(w(1-(1+i))-2i(1+i)),\quad \hbox{if} \quad w\in \Sigma_{c},  \\
\dot{w} =-(w+2i)(w(1-(-0.6-0.4i))-2i(-0.6-0.4i)),\quad \hbox{if} \quad w\in\Sigma_{-}.
\end{cases}
\end{eqnarray}

Note that, in each region, it is given by
\begin{eqnarray*}
    \dot{w}=\pm (w+2i)(w(1-z_{0}^{j})-2iz_{0}^{j}),
\end{eqnarray*}
where  $z_{0}^{\sigma }=x_{0}^{\sigma}+iy_{0}^{\sigma}$, depends on the region considered $z_{0}^{+}=-0.2-0.6i$, $z_{0}^{C}=1+i$ and $z_{0}^{-}=-0.6-0.4i$. Also note that the solutions of previous equation are contained in the level curves of the function
\begin{eqnarray*}  
    H^{\sigma }(x,y)=\dfrac{(x(1-x_{0}^{\sigma })+y_{0}^{\sigma }(2+y))^{2}+((1-x_{0}^{\sigma })y-y_{0}^{\sigma }x-2x_{0}^{\sigma })^{2}}{x^{2}+(y+2)^{2}},
\end{eqnarray*}
for $\sigma =+,c,-$.

Now, we look for the limit cycles that of system (\ref{exemplo.b2.teof}) that intersect the  line $\Im(w)=-1$ at the points $(s_{1},-1)$ and $(s_{2},-1)$, with $s_{1}\neq s_{2}$, and intersects the line $\Im (w)=1$ at the points $(t_{1},1)$ and $(t_{2},1)$, with $t_{1}\neq t_{2}$. Then, these points must satisfy
\begin{eqnarray}\label{sistemateof.b}
   \begin{cases}
   
   (s_{1} - s_{2}) \left( s_{2} (-0.6) - (-0.4) + s_{1} \left( -0.6 + s_{2} (-0.4) \right) \right) = 0,\\
   2 + (6 + t_{2}^2)  + s_{2} (9 + t_{2}^2)  - (t_{2}  + s_{2}^2 (-2 + 3  + t_{2} )) = 0,\\
    (t_1 - t_2) \left( -2.6\, t_2 + 5.4 + t_1(-2.6 - 0.6\, t_2) \right) = 0,\\ 
    2 + (6 + t_{1}^2)  + s_{1} (9 + t_{1}^2)  - (t_{1}  + s_{1}^2 (-2 + 3  + t_{1} )) = 0.
   \end{cases} 
\end{eqnarray}

Since $s_{1}\neq s_{2}$ and $t_{1}\neq t_{2}$, from the first and third equations of (\ref{sistemateof.b}), we obtain that
\begin{eqnarray*}
    s_{2} = \frac{0.6s_{1} - 0.4}{-0.6 - 0.4s_{1}}\quad \hbox{and} \quad t_{2} = \frac{2.6 t_{1} - 5.4}{-2.6 - 0.6 t_{1}}.
\end{eqnarray*}

By substituting the expressions in the second equation of (\ref{sistemateof.b}), we obtain the following system
\begin{align*}
\begin{cases}
2 + (6 + t_1^2) + s_1 (9 + t_1^2) - \left( t_1 + s_1^2 (1 + t_1 (1)) \right) = 0, \\
2 + \left( 6 + \left( \frac{2.6 t_1 - 5.4}{-2.6 - 0.6 t_1} \right)^2 \right) + \frac{0.6 s_1 - 0.4}{-0.6 - 0.4 s_1} \left( 9 + \left( \frac{2.6 t_1 - 5.4}{-2.6 - 0.6 t_1} \right)^2 \right) \\
\quad - \left( \frac{2.6 t_1 - 5.4}{-2.6 - 0.6 t_1} + \left( \frac{0.6 s_1 - 0.4}{-0.6 - 0.4 s_1} \right)^2 \left( 1 + \frac{2.6 t_1 - 5.4}{-2.6 - 0.6 t_1} \right) \right) = 0.
\end{cases}
\end{align*}

This system admits the solutions
\begin{eqnarray*}
    (p_{1},q_{1})=(-1.260240290,-2.5680344499)\; \hbox{and} \; (p_{2},q_{2})=(12.05523832,11.40211976)\\
     (\tilde{p}_{1},\tilde{q}_{1})=(-1.128596670,-1.6659961088)\; \hbox{and}\;  (\tilde{p}_{2},\tilde{q}_{2})=(7.25059468,6.080714625).
\end{eqnarray*}

Therefore, system (\ref{exemplo.b2.teof}) has two limit cycles, one passing though the points $(p_{1},-1)$, $(p_{2},-1)$, $(q_{2},1)$ and $(q_{1},1)$, and another passing through $(\tilde{p}_{1},-1)$, $(\tilde{p}_{2},-1), (\tilde{p}_{2},1)$ and $(\tilde{p}_{1},1)$. Thus, system (\ref{exemplo.b1.teof}) also has two limit cycles.
 \begin{figure}[!htb]
\centering
\includegraphics[scale=0.6]{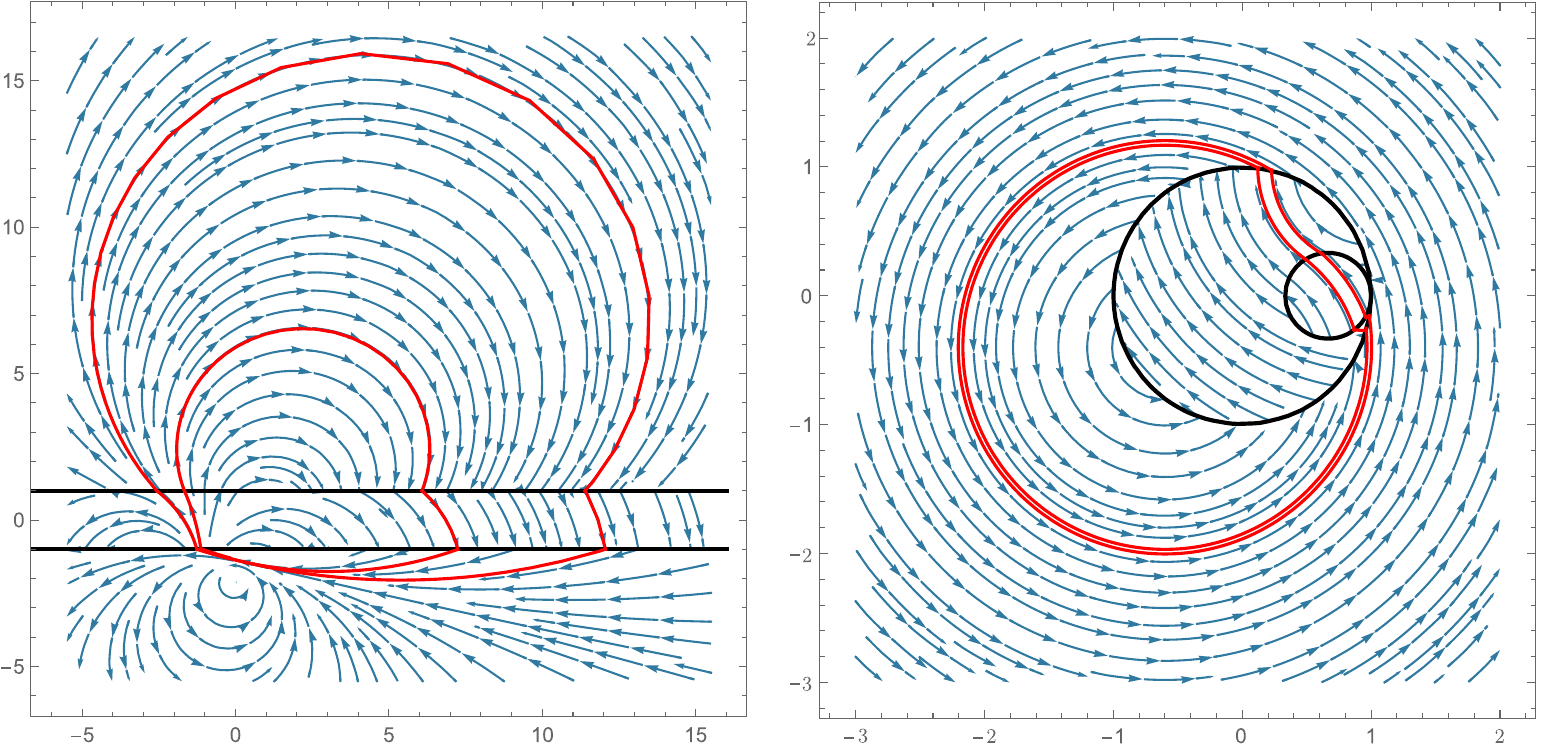}
\caption{Phase portrait of systems  (\ref{exemplo.b2.teof}) and (\ref{exemplo.b1.teof}).}
\label{imga11}
\end{figure}
\end{proof}
\section{Acknowledgments}
The first author is supported by CAPES grand 88887.950547/2024-00. The second author is partially supported by FAPESP grant 2023/02959-5 and 2024/15612-6, CNPq grant
302154/2022-1, and ANR-23-CE40-0028.

\end{document}